\documentclass[11pt,a4paper,leqno]{amsart}

\usepackage[latin1]{inputenc}
\usepackage[T1]{fontenc}
\usepackage{amsfonts}
\usepackage{amsmath}
\usepackage{amssymb}
\usepackage{eurosym}
\usepackage{mathrsfs}
\usepackage{palatino}
\usepackage{color}
\usepackage{esint}
\usepackage{url}
\usepackage{verbatim}

\usepackage{enumerate}

\usepackage[pagebackref,hypertexnames=false, colorlinks, citecolor=blue, linkcolor=blue, urlcolor=red]{hyperref}

\newcommand{\R}{\mathbb{R}}
\newcommand{\C}{\mathbb{C}}

\newcommand{\N}{\mathbb{N}}

\newcommand{\Z}{\mathbb{Z}}
\newcommand{\E}{\mathbb{E}}

\newcommand{\scrF}{\mathscr{F}}

\newcommand{\bla}{\big \langle}
\newcommand{\bra}{\big \rangle}

\numberwithin{equation}{section}

\newcommand{\ud}[0]{\,\mathrm{d}}

\newcommand{\esssup}[0]{\operatornamewithlimits{ess\,sup}}

\newcommand{\ave}[1]{\langle #1\rangle}

\newcommand{\BMO}[0]{\operatorname{BMO}}

\renewcommand{\Re}[0]{\operatorname{Re}}

\newcommand{\ch}[0]{\operatorname{ch}}

\newcommand{\calD}[0]{\mathcal{D}}

\newcommand{\wt}[1]{{\widetilde{#1}}}

\swapnumbers
\theoremstyle{plain}
\newtheorem{thm}[equation]{Theorem}
\newtheorem{lem}[equation]{Lemma}
\newtheorem{prop}[equation]{Proposition}

\theoremstyle{definition}
\newtheorem{defn}[equation]{Definition}

\theoremstyle{remark}
\newtheorem{rem}[equation]{Remark}

\title{Genuinely multilinear weighted estimates for singular integrals in product spaces}

\author{Kangwei Li}
\author{Henri Martikainen}
\author{Emil Vuorinen}

\address[K.L.]{Center for Applied Mathematics, Tianjin University, Weijin Road 92, 300072 Tianjin, China}
\email{kli@tju.edu.cn}

\address[H.M. \& E.V.]{Department of Mathematics and Statistics, University of Helsinki, P.O.B. 68, FI-00014 University of Helsinki, Finland}
\email{henri.martikainen@helsinki.fi}
\email{emil.vuorinen@helsinki.fi}

\makeatletter
\@namedef{subjclassname@2010}{%
  \textup{2010} Mathematics Subject Classification}
\makeatother


\begin{document}

\allowdisplaybreaks

\begin{abstract}
We prove genuinely multilinear weighted estimates for singular integrals in product spaces.
The estimates complete the qualitative weighted theory in this setting.
Such estimates were previously known only in the one-parameter situation. 
Extrapolation gives powerful applications -- for example, a free access to mixed-norm
estimates in the full range of exponents.
\end{abstract}

\maketitle

\section{Introduction}
For given exponents $1 < p_1, \ldots, p_n < \infty$ and $1/p = \sum_i 1/p_i> 0$, a natural form of a weighted estimate in the $n$-variable context has the form
$$
\Big \|g \prod_{i=1}^n w_i \Big \|_{L^p} \lesssim \prod_{i=1}^n \|f_i w_i\|_{L^{p_i}}
$$
for some functions $f_1, \ldots, f_n$ and $g$. It is natural to initially assume that $w_i^{p_i} \in A_{p_i}$, where
$A_q$ stands for the classical Muckenhoupt weights. Even with this assumption the target weight only satisfies $\prod_{i=1}^n w_i^p \in A_{np} \supsetneq A_p$ making
the case $n \ge 2$ have a different flavour than the classical case $n=1$. Importantly, it turns out to be
very advantageous -- we get to the application later -- to only impose a weaker \emph{joint} condition on the tuple of weights $\vec w = (w_1, \ldots, w_n)$ rather
than to assume individual conditions on the weights $w_i^{p_i}$. This gives the problem a genuinely multilinear nature.
For many fundamental mappings $(f_1, \ldots, f_n) \mapsto g(f_1, \ldots, f_n)$, such as the $n$-linear maximal function,
these joint conditions on the tuple $\vec w$ are necessary and sufficient for the weighted bounds.

Genuinely multilinear weighted estimates were first proved for $n$-linear \emph{one-parameter} singular integral operators (SIOs)
by Lerner, Ombrosi, P\'erez, Torres and Trujillo-Gonz\'alez in the extremely influential paper \cite{LOPTT}.
A basic model of an $n$-linear SIO $T$ in $\R^d$ is obtained by setting
\begin{equation*}\label{eq:multilinHEUR}
T(f_1,\ldots, f_n)(x) = U(f_1 \otimes \cdots \otimes f_n)(x,\ldots,x), \qquad x \in \R^d,\, f_i \colon \R^d \to \C,
\end{equation*}
where $U$ is a linear SIO in $\R^{nd}$. See e.g. Grafakos--Torres \cite{GT} for the basic theory.
Estimates for SIOs play a fundamental role in pure and applied analysis -- 
for example, $L^p$ estimates for the homogeneous fractional derivative $D^{\alpha} f=\mathcal F^{-1}(|\xi|^{\alpha} \widehat f(\xi))$ of a product of two or more functions, the \emph{fractional Leibniz rules}, are used in the area of dispersive equations, see e.g. Kato--Ponce \cite{KP} and Grafakos--Oh \cite{GO}.

In the usual one-parameter context of \cite{LOPTT} there is a general philosophy that the maximal function controls SIOs $T$ -- in fact, we have
the concrete estimate
\begin{equation}\label{eq:CF}
\|T(f_1, \ldots, f_n)w\|_{L^p}  \lesssim \|M(f_1, \ldots, f_n)w\|_{L^p}, \qquad p > 0, \, w^p \in A_{\infty}.
\end{equation}
Thus, the heart of the matter of \cite{LOPTT} reduces to the maximal function $$M(f_1, \ldots, f_n) = \sup_I 1_I \prod_{i=1}^n \langle |f_i|\rangle_I,$$ where  $\langle |f_i|\rangle_I
= \fint_I |f_i| = \frac{1}{|I|} \int_I |f_i|$ and the supremum is over cubes $I \subset \R^d$.

In this paper we prove genuinely multilinear weighted estimates for multi-parameter SIOs
in the product space $\R^d = \prod_{i=1}^m \R^{d_i}$. For the classical linear multi-parameter theory and some of its original applications
see e.g. \cite{CF1, CF2, RF1, RF2, RF3, FS, FL, Jo}.  Multilinear multi-parameter estimates arise naturally in applications whenever
a multilinear phenomena, like the fractional Leibniz rules, are combined with product type estimates, such as those that arise when we want to take different partial fractional derivatives $D^{\alpha}_{x_1} D^{\beta}_{x_2} f$. We refer to our recent work \cite{LMV} for a thorough general background on the subject.

It is already known \cite{GLTP} that the multi-parameter maximal function  $(f_1, \ldots, f_n) \mapsto \sup_R 1_R \prod_{i=1}^n \langle |f_i|\rangle_R$, where the supremum is over rectangles $R = \prod_{i=1}^m I^i \subset  \prod_{i=1}^m \R^{d_i}$ with sides parallel to the axes, satisfies the desired genuinely multilinear weighted estimates.
However, in contrast to the one-parameter case, there is no known general principle which would automatically imply the corresponding weighted estimate
for multi-parameter SIOs from the maximal function estimate. In particular, no estimate like \eqref{eq:CF} is known.
In the paper \cite{LMV} we developed the general theory of bilinear bi-parameter SIOs including weighted estimates under the more
restrictive assumption $w_i^{p_i} \in A_{p_i}$. In fact, we only reached these weighted estimates without any additional cancellation assumptions of the type $T1=0$ in \cite{ALMV}.

There are no genuinely multilinear weighted estimates for \textbf{any} multi-parameter SIOs in the literature -- not even for the bi-parameter analogues (see e.g. \cite[Appendix A]{LMV})
of Coifman--Meyer \cite{CM} type multilinear multipliers. Almost ten years after the maximal function result \cite{GLTP} we establish these missing bounds -- not only for some special SIOs -- but for a very general class of $n$-linear $m$-parameter SIOs. With weighted bounds previously being known both in the linear multi-parameter setting \cite{RF1, RF2, HPW} and
in the multilinear one-parameter setting \cite{LOPTT}, we finally establish a holistic view completing the theory of qualitative weighted estimates in the joint
presence of multilinearity and product space theory. 

With the understanding that a Calder\'on--Zygmund operator (CZO) is an SIO
satisfying natural $T1$ type assumptions, our main result reads as follows.
\begin{thm}\label{thm:intro1}
Suppose $T$ is an $n$-linear $m$-parameter CZO in $\R^d = \prod_{i=1}^m \R^{d_i}$.
If $1 < p_1, \ldots, p_n \le \infty$ and $1/p = \sum_{i=1}^n 1/p_i> 0$, we have
$$
\|T(f_1, \ldots, f_n)w \|_{L^p} \lesssim \prod_{i=1}^n \|f_i w_i\|_{L^{p_i}}, \qquad w = \prod_{i=1}^n w_i, 
$$
for all $n$-linear $m$-parameter weights $\vec w = (w_1, \ldots, w_n) \in A_{\vec p}$, $\vec p = (p_1, \ldots, p_n)$. Here $\vec w \in A_{\vec p}$ if
$$
[\vec{w}]_{A_{\vec p}}
:=\sup_R \, \ave{w^p}_R^{\frac 1 p} \prod_{i=1}^n \ave{w_i^{-p_i'}}_R^{\frac 1{p_i'}} < \infty,
$$
where the supremum is over all rectangles $R \subset \R^d$.
\end{thm}
\noindent For the exact definitions, see the main text.

Recent extrapolation methods are crucial both for the proof and for the applications.
The extrapolation theorem of Rubio de Francia says that if $\|g\|_{L^{p_0}(w)} \lesssim \|f\|_{L^{p_0}(w)}$ for some $p_0 \in (1,\infty)$ and all $w \in A_{p_0}$, then $\|g\|_{L^{p}(w)} \lesssim \|f\|_{L^{p}(w)}$ for all $p \in (1,\infty)$ and all $w \in A_{p}$.
In \cite{GM} (see also \cite{DU}) a multivariable analogue was developed in the setting $w_i^{p_i} \in A_{p_i}$, $i = 1, \ldots, n$. Such extrapolation results are already of fundamental use in proving other estimates -- often just to even deduce the full $n$-linear range of unweighted estimates $\prod_{j=1}^n L^{p_j}\to L^p$, $\sum_j 1/p_j = 1/p$, $1 < p_j < \infty$, $1/n <p < \infty$,
from some particular single tuple $(p_1, \ldots, p_n, p)$. Indeed, reaching $p \le 1$ can often be a crucial challenge, particularly so in multi-parameter settings where
many other tools are completely missing.

Very recently, in \cite{LMO} it was shown that also the genuinely multilinear weighted estimates can be extrapolated. In the subsequent paper \cite{LMMOV} (see also \cite{Nieraeth}) a key advantage
of extrapolating using the general weight classes was identified: it is possible to both start the extrapolation and, importantly, to reach -- as a consequence of the extrapolation -- weighted
estimates with $p_i = \infty$. See Theorem \ref{thm:ext} for a formulation of these general extrapolation principles. Moreover, extrapolation
is flexible in the sense that one can extrapolate both $1$-parameter and $m$-parameter, $m\ge 2$, weighted estimates.

These new extrapolation results are extremely useful e.g. in proving mixed-norm estimates -- for example, in the bi-parameter case they yield that
\[
\| T(f_1,\ldots, f_n)\|_{L^p(\R^{d_1}; L^q(\R^{d_2}))}\lesssim \prod_{i=1}^n \| f_i\|_{L^{p_i}(\R^{d_1}; L^{q_i}(\R^{d_2}))},
\]
where $1<p_i, q_i\le \infty$, $\frac 1p= \sum_i \frac{1}{p_i} >0$ and $\frac 1q=\sum_i \frac{1}{q_i} >0$. The point is that even all of the various cases involving $\infty$ become immediate. See e.g. \cite{DO, LMMOV, LMV} for some of the previous mixed-norm estimates.
Compared to \cite{LMMOV} we can work with completely general $n$-linear $m$-parameter SIOs instead of bi-parameter tensor products of $1$-parameter SIOs, and the
proof is much simplified due to the optimal weighted estimates, Theorem \ref{thm:intro1}.

We also use extrapolation to give a new short proof of the boundedness of the multi-parameter $n$-linear maximal function \cite{GLTP} -- see Proposition \ref{prop:prop2}. 

On the technical level there is no existing approach to our result: the modern one-parameter tools (such as sparse domination in the multilinear setting, see e.g. \cite{CUDPOU})
are missing and many of the bi-parameter methods \cite{LMV} used in conjunction with the assumption that each weight individually satisfies $w_i^{p_i} \in A_{p_i}$ are of little use.
Aside from maximal function estimates, multi-parameter estimates require various square function estimates (and combinations of maximal function and square function estimates). Similarly as
one cannot use $\prod_i Mf_i$ instead of $M(f_1, \ldots, f_n)$ due to the nature of the multilinear weights,
it is also not possible to use classical square function estimates separately for the functions $f_i$. Now, this interplay
makes it impossible to decouple estimates to terms like $\|M f_1 \cdot w_1\|_{L^{p_1}} \|S f_2 \cdot w_2\|_{L^{p_2}}$, since
neither of them would be bounded separately as $w_1^{p_1} \not \in A_{p_1}$ and $w_2^{p_2} \not \in A_{p_2}$. However,
such decoupling of estimates has previously seemed almost indispensable. 

Our proof starts with the reduction to dyadic model operators \cite{AMV} (see also \cite{DLMV2, Hy1, LMOV, Ma1, Ou}), which
is a standard idea. After this we introduce a family of $n$-linear multi-parameter square function type objects $A_k$. On the idea level, a big part of the proof works by taking a dyadic model operator $S$
and finding an appropriate square function $A_k$ so that
$$
\|S(f_1, \ldots, f_n)w \|_{L^p} \lesssim \|A_k(f_1, \ldots, f_n)w \|_{L^p}.
$$
This requires different tools depending on the model operator in question and is a new way to estimate model operators that respects the $n$-linear structure fully.
We then prove that all of our operators $A_k$ satisfy the genuinely $n$-linear weighted estimates
$$
\|A_{k}(f_1, \ldots, f_n)w \|_{L^p} \lesssim \prod_{i=1}^n \|f_i w_i\|_{L^{p_i}}.
$$
This is done with an argument that is based on using duality and lower square function estimates iteratively until all of the cancellation present
in these square functions has been exploited.

Aside from the full range of mixed-norm estimates, the weighted estimates immediately give other applications as well.
We present here a result on commutators, which greatly generalises \cite{LMV}.
Commutator estimates appear all over analysis implying e.g.
factorizations for Hardy functions \cite{CRW}, certain div-curl lemmas relevant in compensated compactness, and were recently connected to
the Jacobian problem $Ju = f$ in $L^p$ (see \cite{HyCom}). For a small sample of commutator estimates in various other key setting see e.g.
\cite{FL,  HLW, HPW, LOR1}.
\begin{thm}
Suppose $T$ is an $n$-linear $m$-parameter CZO in $\R^d = \prod_{i=1}^m \R^{d_i}$,
$1 < p_1, \ldots, p_n \le \infty$ and $1/p = \sum_{i=1}^n 1/p_i> 0$.
Suppose also that $\|b\|_{\operatorname{bmo}} = \sup_R \frac{1}{|R|}
\int_R |b-\ave{b}_R| < \infty$.
Then for all $1\le k\le n$ we have the commutator estimate
\begin{equation*}
\begin{split}
\| [b, T]_k(f_1,\ldots, f_n) w\|_{L^p} &\lesssim \|b\|_{\operatorname{bmo}} \prod_{i=1}^n \|f_iw_i\|_{L^{p_i}}, \\
[b, T]_k(f_1,\ldots, f_n) &:= bT(f_1, \ldots, f_n) - T(f_1, \ldots, f_{k-1}, bf_k, f_{k+1}, \ldots, f_n),
\end{split}
\end{equation*}
for all $n$-linear $m$-parameter weights $\vec w = (w_1, \ldots, w_n) \in A_{\vec p}$.
Analogous results hold for iterated commutators.
\end{thm}
We note that we can also finally dispose of some of the sparse domination tools that restricted some of the theory of \cite{LMV} to bi-parameter.

\subsection*{Acknowledgements}
K. Li was supported by the National Natural Science Foundation of China through project number 12001400.
H. Martikainen and E. Vuorinen were supported by the Academy of Finland through project numbers 294840 (Martikainen) and 327271 (Martikainen, Vuorinen), and by the three-year research grant 75160010 of the University of Helsinki. 

The authors thank the anonymous referee for careful reading of the paper and for suggestions which improve the readability of the paper.

\section{Preliminaries}\label{sec:def}

Throughout this paper $A\lesssim B$ means that $A\le CB$ with some constant $C$ that we deem unimportant to track at that point.
We write $A\sim B$ if $A\lesssim B\lesssim A$. Sometimes we  e.g. write $A \lesssim_{\epsilon} B$ if we want to make the point that $A \le C(\epsilon) B$.

\subsection{Dyadic notation}
Given a dyadic grid $\calD$ in $\R^d$, $I \in \calD$ and $k \in \Z$, $k \ge 0$, we use the following notation:
\begin{enumerate}
\item $\ell(I)$ is the side length of $I$.
\item $I^{(k)} \in \calD$ is the $k$th parent of $I$, i.e., $I \subset I^{(k)}$ and $\ell(I^{(k)}) = 2^k \ell(I)$.
\item $\ch(I)$ is the collection of the children of $I$, i.e., $\ch(I) = \{J \in \calD \colon J^{(1)} = I\}$.
\item $E_I f=\langle f \rangle_I 1_I$ is the averaging operator, where $\langle f \rangle_I = \fint_{I} f = \frac{1}{|I|} \int _I f$.
\item $\Delta_If$ is the martingale difference $\Delta_I f= \sum_{J \in \ch (I)} E_{J} f - E_{I} f$.
\item $\Delta_{I,k} f$ is the martingale difference block
$$
\Delta_{I,k} f=\sum_{\substack{J \in \calD \\ J^{(k)}=I}} \Delta_{J} f.
$$
\end{enumerate}

For an interval $J \subset \R$ we denote by $J_{l}$ and $J_{r}$ the left and right
halves of $J$, respectively. We define $h_{J}^0 = |J|^{-1/2}1_{J}$ and $h_{J}^1 = |J|^{-1/2}(1_{J_{l}} - 1_{J_{r}})$.
Let now $I = I_1 \times \cdots \times I_d \subset \R^d$ be a cube, and define the Haar function $h_I^{\eta}$, $\eta = (\eta_1, \ldots, \eta_d) \in \{0,1\}^d$, by setting
\begin{displaymath}
h_I^{\eta} = h_{I_1}^{\eta_1} \otimes \cdots \otimes h_{I_d}^{\eta_d}.
\end{displaymath}
If $\eta \ne 0$ the Haar function is cancellative: $\int h_I^{\eta} = 0$. We exploit notation by suppressing the presence of $\eta$, and write $h_I$ for some $h_I^{\eta}$, $\eta \ne 0$. Notice that for $I \in \calD$ we have $\Delta_I f = \langle f, h_I \rangle h_I$ (where the finite $\eta$ summation is suppressed), $\langle f, h_I\rangle := \int fh_I$.

We make a few clarifying comments related to the use of Haar functions. In the model operators coming from the representation theorem
there are Haar functions involved. There we use the just mentioned convention that $h_I$ means some unspecified cancellative Haar function $h_I^{\eta}$ which we
do not specify. On the other hand, the square function estimates in Section \ref{sec:SquareFunctions} are formulated using martingale differences
(which involve multiple Haar functions as $\Delta_I f = \sum_{\eta \ne 0} \langle f, h_I^{\eta} \rangle h_I^{\eta}$).
When we estimate the model operators, we carefully consider this difference by
passing from the Haar functions into martingale differences via the simple identity
\begin{equation}\label{eq:HaarMart}
\langle f, h_I \rangle=\langle \Delta_I f, h_I \rangle,
\end{equation} 
which follows from $\Delta_I f=\sum_{\eta \not=0} \langle f, h^{\eta}_I \rangle h^{\eta}_I$
and orthogonality. 

\subsection{Multi-parameter notation}
We will be working on the $m$-parameter product space $\R^d = \prod_{i=1}^m \R^{d_i}$.
We denote a general dyadic grid in $\R^{d_i}$ by $\calD^i$. We denote cubes in $\calD^i$ by $I^i, J^i, K^i$, etc.
Thus, our dyadic rectangles take the forms $\prod_{i=1}^m I^i$, $\prod_{i=1}^m J^i$, $\prod_{i=1}^m K^i$ etc. We usually
denote the collection of dyadic rectangles by $\calD = \prod_{i=1}^m \calD^i$.

If $A$ is an operator acting on $\R^{d_1}$, we can always let it act on the product space $\R^d$ by setting $A^1f(x) = A(f(\cdot, x_2, \ldots, x_n))(x_1)$. Similarly, we use the notation $A^i f$ if $A$ is originally an operator acting on $\R^{d_i}$. Our basic multi-parameter dyadic operators -- martingale differences and averaging operators -- are obtained by simply chaining together relevant one-parameter operators. For instance, an $m$-parameter martingale difference is
$$
\Delta_R f = \Delta_{I^1}^1 \cdots \Delta_{I^m}^m f, \qquad R = \prod_{i=1}^m I^i.
$$
When we integrate with respect to only one of the parameters we may e.g. write
\[
\langle f, h_{I^1} \rangle_1(x_2, \ldots, x_n):=\int_{\R^{d_1}} f(x_1, \ldots, x_n)h_{I^1}(x_1) \ud x_1
\]
or
$$
\langle f \rangle_{I^1, 1}(x_2, \ldots, x_n) := \fint_{I^1} f(x_1, \ldots, x_n) \ud x_1.
$$

\subsection{Adjoints}\label{sec:adjoints}
Consider an $n$-linear operator $T$ on $\R^d = \R^{d_1} \times \R^{d_2}$. Let $f_j = f_j^1 \otimes f_j^2$, $j = 1, \ldots, n+1$.
We set up notation for the adjoints of $T$ in the bi-parameter situation. We let $T^{j*}$, $j \in \{0, \ldots, n\}$, denote the full adjoints, i.e., 
$T^{0*} = T$ and otherwise
$$
\langle T(f_1, \dots, f_n), f_{n+1} \rangle
= \langle T^{j*}(f_1, \dots, f_{j-1}, f_{n+1}, f_{j+1}, \dots, f_n), f_j \rangle.
$$
A subscript $1$ or $2$ denotes a partial adjoint in the given parameter -- for example, we define
$$
\langle T(f_1, \dots, f_n), f_{n+1} \rangle
= \langle T^{j*}_1(f_1, \dots, f_{j-1}, f_{n+1}^1 \otimes f_j^2, f_{j+1}, \dots, f_n), f_j^1 \otimes f_{n+1}^2 \rangle.
$$
Finally, we can take partial adjoints with respect to different parameters in different slots also -- in that case we denote the adjoint by $T^{j_1*, j_2*}_{1,2}$. It simply interchanges
the functions $f_{j_1}^1$ and $f_{n+1}^1$ and the functions $f_{j_2}^2$ and $f_{n+1}^2$. Of course, we e.g. have $T^{j*, j*}_{1,2} = T^{j*}$ and $T^{0*, j*}_{1,2} = T^{j*}_{2}$,
so everything can be obtained, if desired, with the most general notation $T^{j_1*, j_2*}_{1,2}$.
In any case, there are $(n+1)^2$ adjoints (including $T$ itself). These notions have obvious extensions to $m$-parameters.

\subsection{Structure of the paper}
To avoid unnecessarily complicating the notation, we start by proving everything in the bi-parameter case $m=2$. Importantly, we present a proof which does not exploit this in a way that would not be extendable to $m$-parameters (e.g., our proof for the partial paraproducts does not exploit sparse domination for the
appearing one-parameter paraproducts). At the end, we demonstrate for some key model operators how the general case can be dealt with.

\section{Weights}\label{sec:weights}
The following notions have an obvious extension to $m$-parameters.
A weight $w(x_1, x_2)$ (i.e. a locally integrable a.e. positive function) belongs to the bi-parameter weight class $A_p = A_p(\R^{d_1} \times \R^{d_2})$, $1 < p < \infty$, if
$$
[w]_{A_p} := \sup_{R}\, \ave{w}_R  \ave{w^{1-p'}}^{p-1}_R
=  \sup_{R}\, \ave{w}_R  \ave{w^{-\frac{1}{p-1}}}^{p-1}_R
 < \infty,
$$
where the supremum is taken over rectangles $R$ -- that is, over $R = I^1 \times I^2$ where $I^i \subset \R^{d_i}$ is a cube. Thus, this is the one-parameter definition but cubes are
replaced by rectangles.

We have
\begin{equation}\label{eq:eq28}
[w]_{A_p(\R^{d_1} \times \R^{d_2})} < \infty \textup { iff } \max\big( \esssup_{x_1 \in \R^{d_1}} \,[w(x_1, \cdot)]_{A_p(\R^{d_2})}, \esssup_{x_2 \in \R^{d_2}}\, [w(\cdot, x_2)]_{A_p(\R^{d_1})} \big) < \infty,
\end{equation}
and that $$
\max\big( \esssup_{x_1 \in \R^{d_1}} \,[w(x_1, \cdot)]_{A_p(\R^{d_2})}, \esssup_{x_2 \in \R^{d_2}}\, [w(\cdot, x_2)]_{A_p(\R^{d_1})} \big) \le [w]_{A_p(\R^{d_1}\times \R^{d_2})},
$$
while the constant $[w]_{A_p}$ is dominated by the maximum to some power.
It is also useful that $\ave{w}_{I^2,2} \in A_p(\R^{d_1})$ uniformly on the cube $I^2 \subset \R^{d_2}$.
 For basic bi-parameter weighted theory see e.g. \cite{HPW}.
We say $w\in A_\infty(\R^{d_1}\times \R^{d_2})$ if
\[
[w]_{A_\infty}:=\sup_R \, \ave{w}_R \exp\big( \ave{\log w^{-1}}_R  \big)<\infty.
\]
It is well-known that
$$A_\infty(\R^{d_1}\times \R^{d_2})=\bigcup_{1<p<\infty}A_p(\R^{d_1}\times \R^{d_2}).$$
We also define
$$
[w]_{A_1} = \sup_R \, \ave{w}_R \esssup_R w^{-1}.
$$

We introduce the classes of multilinear Muckenhoupt weights that we will use.
\begin{defn}\label{defn:defn1}
Given $\vec p=(p_1, \ldots,  p_n)$ with $1 \le p_1, \ldots, p_n \le \infty$ we say that 
$\vec{w}=(w_1, \ldots, w_n)\in A_{\vec p} = A_{\vec p}(\R^{d_1} \times \R^{d_2})$, if
$$
0<w_i <\infty, \qquad i = 1, \ldots, n,
$$
almost everywhere and
$$
[\vec{w}]_{A_{\vec p}}
:=\sup_R \, \ave{w^p}_R^{\frac 1 p} \prod_{i=1}^n \ave{w_i^{-p_i'}}_R^{\frac 1{p_i'}} < \infty,
$$
where the supremum is over rectangles $R$,
$$
w := \prod_{i=1}^n w_i \qquad \textup{and} \qquad \frac 1 p = \sum_{i=1}^n \frac 1 {p_i}.
$$
If $p_i = 1$ we interpret $\ave{w_i^{-p_i'}}_R^{\frac 1{p_i'}}$ as
$\esssup_R w_i^{-1}$, and if $p = \infty$ we interpret 
$\ave{w^p}_R^{\frac 1 p}$ as $\esssup_R w$.
\end{defn}
\begin{rem}
\begin{enumerate}
\item It is important that the lower bound
\begin{equation}\label{eq:eq7}
 \ave{w^p}_R^{\frac 1 p} \prod_{i=1}^n \ave{w_i^{-p_i'}}_R^{\frac 1{p_i'}} \ge 1
\end{equation}
holds always. To see this recall that for $\alpha_1, \alpha_2 > 0$ we have by H\"older's inequality that
\begin{equation}\label{eq:eq8}
\begin{split}
1 \le \ave{ w^{-\alpha_1}}_R^{\frac{1}{\alpha_1}}  \ave{ w^{\alpha_2}}_R^{\frac{1}{\alpha_2}}.
\end{split}
\end{equation}
Apply this with $\alpha_2 = p$ and $\alpha_1 = \frac{1}{n-\frac{1}{p}}$. Then apply H\"older's inequality with the exponents
$u_i = \Big(n-\frac{1}{p}\Big)p_i'$ to get $\Big \langle \Big( \prod_{i=1}^n w_i \Big)^{-\frac{1}{n-\frac{1}{p}}} \Big \rangle_R^{n-\frac{1}{p}} \le \prod_{i=1}^n \ave{ w_i^{-p_i'}}_R^{\frac{1}{p_i'}}$.
\item Our definition is essentially the usual one-parameter definition \cite{LOPTT} with the difference that cubes are replaced by rectangles. However, we are also using the renormalised definition
from \cite{LMMOV} that works better with the exponents $p_i = \infty$. Compared to the usual formulation of \cite{LOPTT} the relation is that
$[w_1^{p_1}, \ldots, w_n^{p_n}]_{A_{\vec p}}^{\frac 1p}$ with $A_{\vec p}$ defined as in \cite{LOPTT} agrees with our $[\vec w]_{A_{\vec p}}$ when $p_i < \infty$.
\item The case $p_1 = \cdots = p_n = \infty = p$ can be used as the starting point of extrapolation. This is rarely useful but we will find use for it when we consider the multilinear
maximal function.
\end{enumerate}
\end{rem}

The following characterization of the class $ A_{\vec p}$ is convenient. The one-parameter result with the different normalization is \cite[Theorem 3.6]{LOPTT}. We record the proof
for the convenience of the reader.
\begin{lem}\label{lem:lem1}
Let $\vec p=(p_1, \ldots,  p_n)$ with $1 \le p_1, \ldots, p_n \le \infty$, $1/p = \sum_{i=1}^n 1/p_i \ge 0$,
$\vec{w}=(w_1, \ldots, w_n)$ and $w = \prod_{i=1}^n w_i$.
We have
$$
[w_i^{-p_i'}]_{A_{np_i'}} \le [\vec{w}]_{A_{\vec p}}^{p_i'}, \qquad i = 1, \ldots, n,
$$
and
$$
[w^p]_{A_{np}} \le  [\vec{w}]_{A_{\vec p}}^{p}.
$$
In the case $p_i = 1$ the estimate is interpreted as $[w_i^{\frac 1n}]_{A_1} \le [\vec{w}]_{A_{\vec p}}^{1/n}$, and in the case $p=\infty$
we have $[w^{-\frac{1}{n}}]_{A_1} \le [\vec{w}]_{A_{\vec p}}^{1/n}$.

Conversely, we have
$$
 [\vec{w}]_{A_{\vec p}} \le [w^p]_{A_{np}}^{\frac{1}{p}} \prod_{i=1}^n [w_i^{-p_i'}]_{A_{np_i'}}^{\frac{1}{p_i'}}.
$$
\end{lem}
\begin{proof}
We fix an arbitrary $j \in \{1, \ldots, n\}$ for which we will show $[w_j^{-p_j'}]_{A_{np_j'}} \le [\vec{w}]_{A_{\vec p}}^{p_j'}$.
Notice that
\begin{equation}\label{eq:eq1}
\frac{1}{p} + \sum_{i \ne j} \frac{1}{p_i'} = n - 1 + \frac{1}{p_j}.
\end{equation}
We define $q_j$ via the identity
$$
\frac{1}{q_j} = \frac{1}{n - 1 + \frac{1}{p_j}}\cdot \frac 1p
$$
and for $i \ne j$ we set
$$
\frac{1}{q_i} = \frac{1}{n - 1 + \frac{1}{p_j}} \cdot\frac{1}{p_i'}.
$$
From \eqref{eq:eq1} we have that $\sum_i \frac{1}{q_i} = 1$.
By definition we have
\begin{equation}\label{eq:eq3}
[w_j^{-p_j'}]_{A_{np_j'}} = \sup_R \, \ave{ w_j^{-p_j'}}_R \ave{ w_j^{p_j' \frac{1}{np_j'-1}}}_R^{np_j' - 1}.
\end{equation}
Notice that
$$
p_j' \frac{1}{np_j'-1} = \frac{1}{n-\frac{1}{p_j'}} = \frac{1}{n - 1 + \frac{1}{p_j}}.
$$
Using H\"older's inequality with the exponents $q_1, \ldots, q_n$ we have the desired estimate
$$
\ave{ w_j^{-p_j'}}_R^{\frac 1{p_j'}} \ave{ w_j^{\frac{p}{q_j}}}_R^{\frac{q_j}{p}} = \ave{ w_j^{-p_j'}}_R^{\frac 1{p_j'}}\ave{ w^{\frac{p}{q_j}} \prod_{i \ne j } w_i^{-\frac{p}{q_j}}}_R^{\frac{q_j}{p}} \le \ave{w^p}_R^{\frac{1}{p}  }   \prod_{i} \ave{ w_i^{-p_i'} }_R^{\frac{1}{p_i'} } \le [\vec{w}]_{A_{\vec p}}.
$$
When $p_j=1$ this is 
$\esssup_R w_j^{-1} \ave{ w_j^{\frac{1}{n}}}_R^{n}\le [\vec{w}]_{A_{\vec p}}$, and so $[w_j^{\frac 1n}]_{A_1}^n \le [\vec{w}]_{A_{\vec p}}$.

We now move on to bounding $[w^p]_{A_{np}}$. Notice that by definition
\begin{equation}\label{eq:eq4}
[w^p]_{A_{np}} = \sup_R \, \ave{w^p}_R \ave{ w^{-\frac{p}{np-1}}}_R^{np-1}.
\end{equation}
We define $s_i$ via
$$
-\frac{p}{np-1} \cdot s_i = - p_i'
$$
and notice that then $\sum_i \frac{1}{s_i} = 1$.
Then, by H\"older's inequality with the exponents $s_1, \ldots, s_n$ we have
$$
\ave{w^p}_R \ave{ w^{-\frac{p}{np-1}}}_R^{np-1} \le \ave{w^p}_R \prod_i \ave{  w_i^{-p_i'}}_R^{\big(\frac{p}{np-1}\big)\frac{1}{p_i'}(np-1)} = \Big[ \ave{w^p}_R^{\frac{1}{p}} \prod_i \ave{  w_i^{-p_i'}}_R^{\frac{1}{p_i'}} \Big]^{p} \le [\vec{w}]_{A_{\vec p}}^{p},
$$
which is the desired bound for $[w^p]_{A_{np}}$. Notice that in the case $p=\infty$ we get
$$
[w^{-\frac{1}{n}}]_{A_1}^n = \sup_R \big\langle w^{-\frac{1}{n}} \big \rangle_R^{n} \esssup_R w
\le \sup_R\Big[ \prod_i \langle w_i^{-1} \rangle_R \Big]\esssup_R w = [\vec{w}]_{A_{\vec p}}.
$$

We then move on to bounding $ [\vec{w}]_{A_{\vec p}}$. It is based on the following inequality
\begin{equation}\label{eq:eq2}
1 \le \ave{ w^{-\frac{p}{np-1}} }_R^{n-\frac{1}{p}} \prod_i \Big\langle w_i^{\frac{1}{n - 1 + \frac{1}{p_i}}} \Big\rangle_R^{n-1+\frac{1}{p_i}}.
\end{equation}
Before proving this, we show how it implies the desired bound. We have
\begin{align*}
[\vec{w}&]_{A_{\vec p}} = \sup_R \, \ave{w^p}_R^{\frac 1 p} \prod_i \ave{w_i^{-p_i'}}_R^{\frac 1{p_i'}} \\
&\le \sup_R  \Big[ \ave{w^p}_R  \ave{ w^{-\frac{p}{np-1}} }_R^{np-1} \Big]^{\frac 1 p}
 \prod_i \Big[ \ave{w_i^{-p_i'}}_R \big \langle  w_i^{p_i' \frac{1}{np_i'-1}}\big \rangle_R^{np_i' - 1} \Big]^{\frac 1{p_i'}} 
 \le [w^p]_{A_{np}}^{\frac{1}{p}} \prod_i [w_i^{-p_i'}]_{A_{np_i'}}^{\frac{1}{p_i'}},
\end{align*} 
where in the last estimate we recalled \eqref{eq:eq3} and \eqref{eq:eq4}.

Let us now give the details of \eqref{eq:eq2}. We apply \eqref{eq:eq8} with $\alpha_1 = \frac{p}{np-1}$ and $\alpha_2 = \frac{1}{n(n-1)+\frac{1}{p}}$ to get
$$
1 \le \ave{ w^{-\frac{p}{np-1}} }_R^{n-\frac{1}{p}} \Big\langle w^{ \frac{1}{n(n-1)+\frac{1}{p}} }\Big\rangle_R^{n(n-1)+\frac{1}{p}}.
$$
The first term is already as in \eqref{eq:eq2}. Define $u_i$ via
$$
 \frac{1}{n(n-1)+\frac{1}{p}} u_i = \frac{1}{n - 1 + \frac{1}{p_i}}
$$
and notice that by H\"older's inequality with these exponents ($\sum_i \frac{1}{u_i} = 1$) we have
$$
\Big\langle w^{ \frac{1}{n(n-1)+\frac{1}{p}} }\Big\rangle_R^{n(n-1)+\frac{1}{p}} \le \prod_i \Big\langle w_i^{\frac{1}{n - 1 + \frac{1}{p_i}} }\Big\rangle_R^{ n - 1 + \frac{1}{p_i}},
$$
which matches the second term in \eqref{eq:eq2}.
\end{proof}
The following duality of multilinear weights is handy -- see \cite[Lemma 3.1]{LMS}. We give the short proof for convenience.
\begin{lem}\label{lem:lem7}
Let $\vec p=(p_1, \ldots,  p_n)$ with $1 < p_1, \ldots, p_n < \infty$ and $\frac 1 p = \sum_{i=1}^n \frac 1 {p_i} \in (0,1)$.
Let $\vec{w}=(w_1, \ldots, w_n)\in A_{\vec p}$ with $w = \prod_{i=1}^n w_i$ and define
\begin{align*}
\vec w^{\, i} &= (w_1, \ldots, w_{i-1}, w^{-1}, w_{i+1}, \ldots, w_n), \\
\vec p^{\,i} &= (p_1, \ldots, p_{i-1}, p', p_{i+1}, \ldots, p_n).
\end{align*}
Then we have
$$
[\vec{w}^{\,i}]_{A_{\vec p^{\, i}}} = 
[\vec{w}]_{A_{\vec p}}.
$$
\end{lem}
\begin{proof}
We take $i=1$ for notational convenience. Notice that $\frac{1}{p'} + \sum_{i=2}^n \frac{1}{p_i}  = \frac{1}{p_1'}$.
Notice also that $w^{-1} \prod_{i=2}^n w_i = w_1^{-1}$.
Therefore, we have
$$
[\vec{w}^{\,i}]_{A_{\vec p^{\, i}}} = \ave{w_1^{-p_1'}}_R^{\frac{1}{p_1'}} \ave{ w^{p}}_R^{\frac{1}{p}} \prod_{i=2}^n \ave{ w_i^{-p_i'}}_R^{\frac{1}{p_i'}} = [\vec{w}]_{A_{\vec p}}.
$$
\end{proof}

We now recall the recent extrapolation result of \cite{LMMOV}. The previous version, which did not yet allow exponents to be $\infty$
appeared in \cite{LMO}. For related independent work see \cite{Nieraeth}. The previous extrapolation results with the separate assumptions $w_i^{p_i} \in A_{p_i}$
appear in \cite{GM} and \cite{DU}. An even more general result than the one below appears in \cite{LMMOV}, but we will not need that generality here. Finally, we note that the proof
of this extrapolation result can be made to work in $m$-parameters even though \cite{LMMOV} provides the details only in the one-parameter case
-- we give more details later in Section \ref{app:app1}.
\begin{thm}\label{thm:ext}
	Let $f_1, \ldots, f_n$ and $g$ be given functions.
Given $\vec p=(p_1,\dots, p_n)$ with $1\le p_1,\dots, p_n\le \infty$ let $\frac1p= \sum_{i=1}^n \frac{1}{p_i}$. Assume that given any $\vec w=(w_1,\dots, w_n) \in A_{\vec p}$ the inequality
	\begin{equation}\label{extrapol:H*}
	\|gw\|_{L^{p}} \lesssim  \prod_{i=1}^n \|f_iw_i\|_{L^{p_i} }
	\end{equation}
	holds, where $w:=\prod_{i=1}^n w_i $. Then for all exponents $\vec q=(q_1,\dots,q_n)$, with $1< q_1,\dots, q_n\le \infty$ and $\frac1q= \sum_{i=1}^n \frac{1}{q_i} >0$, and for all weights $\vec v=(v_1,\dots, v_n) \in A_{\vec q}$ the inequality
	\begin{equation*}\label{extrapol:C*}
	\|gv\|_{L^{q} } \lesssim \prod_{i=1}^n \|f_iv_i\|_{L^{q_i} }
	\end{equation*}
	holds, where $v:=\prod_{i=1}^n v_i $. 
	
	Given functions $f_1^j, \ldots, f_n^j$ and $g^j$ so that \eqref{extrapol:H*} holds uniformly on $j$, we have
	for the same family of exponents and
	weights as above, and for all exponents $\vec{s}=(s_1,\dots, s_n)$ with $1< s_1,\dots, s_n\le \infty$ and $\frac1s=\sum_i \frac{1}{s_i} >0$ the inequality  
	\begin{equation}\label{extrapol:vv*}
	\| (g^j v)_j\|_{L^{q}(\ell^s)}
	\lesssim
	\prod_{i=1}^n \|(f_i^j v_i)_j\|_{L^{q_i}(\ell^{s_i}) }.
	\end{equation}
\end{thm}
\begin{rem}
Using Lemma \ref{lem:lem7} and extrapolation, Theorem \ref{thm:ext},
we see that the weighted boundedness of $T$ transfers to the adjoints $T^{j*}$. Partial adjoints
have to always be considered separately, though.
\end{rem}

As a final thing in this section, we demonstrate the necessity of the $A_{\vec p}$ condition for the weighted boundedness of SIOs.
We work in the $m$-parameter setting and let $\R^d=\R^{d_1} \times \dots \times \R^{d_n}$.
Let $R_j$ be the following version of the $n$-linear one-parameter Riesz transform in $\R^{d_j}$:
\[
R_j(f_1,\dots, f_n)=\textup{p.v.} \int_{\R^{d_jn}}\frac{\sum_{i=1}^n\sum_{k=1}^{d_j}(x-y_i)_k}{(\sum_{i=1}^{n}|x-y_i|)^{d_j n+1}}f_1(y_1)\cdots f_n(y_n) \ud y_1\cdots \ud y_n,
\]
where $(x-y_i)_k$ is the $k$-th coordinate of $x-y_i \in \R^{d_j}$. 
Consider the tensor product
$
R_{1}\otimes R_{2}\otimes \cdots \otimes R_{m}. 
$
Let $\vec w =(w_1, \dots, w_n)$ be a multilinear weight, that is, $0 < w_i < \infty$ a.e., and denote $w=\prod_{i=1}^n w_i$. 
Suppose that for some 
exponents $1<  p_1,\dots, p_n\le \infty$ with $1/p=\sum_{i=1}^n 1/{p_i}>0$ the estimate
\[
\| R_{1}\otimes R_{2}\otimes \cdots \otimes R_{m} (f_1, \dots, f_n)\|_{L^{p,\infty}(w^p)}
\lesssim  \prod_{i=1}^n \|f_iw_i\|_{L^{p_i}}
\]
holds for all $f_i \in L^\infty _c$. We show that $\vec w$ is an $m$-parameter $A_{\vec p}$ weight.

Define $\sigma_i=w_i^{-p_i'}$. Let $E \subset  \R^d$ be an arbitrary set such that $1_E \sigma_i \in L^\infty_c$ for all $i=1, \dots, n$. 
Fix an $m$-parameter rectangle $R=R^1 \times \cdots \times \R^m \subset \R^d$, where each $R^j$ is a cube. 
Let $R^+= (R^1)^+ \times \cdots \times (R^m)^+$,
where $(R^j)^+:=R^j+(\ell(R^j), \dots, \ell(R^j))$.

Using the kernel of $R_1 \otimes \cdots \otimes R_m$  we have for all $x \in R^+$ that
\begin{align*}
R_{1}\otimes R_{2}\otimes \cdots \otimes R_{m} (1_E \sigma_11_R, \dots, 1_E\sigma_n1_R)(x)
\gtrsim \prod_{i=1}^n \langle  1_E\sigma_i\rangle_R. 
\end{align*}
Hence 
\begin{equation*}
w^p(R^+)^{\frac 1p}\prod_{i=1}^n \langle 1_E\sigma_i\rangle_R  
\lesssim  \prod_{i=1}^n \|1_E\sigma_i1_Rw_i \|_{L^{p_i}}     
=\prod_{i=1}^n \sigma_i(E \cap R)^{\frac{1}{p_i}},
\end{equation*}
which gives that
$
\langle w^p\rangle_{R^+}^{\frac 1p}\prod_{i=1}^n \langle 1_E\sigma_i\rangle_R^{\frac{1}{p_i'}} 
\lesssim 1. 
$
Since $E$ was arbitrary this implies the estimate
\begin{equation}\label{eq:eq29}
\langle w^p\rangle_{R^+}^{\frac 1p}\prod_{i=1}^n \langle \sigma_i\rangle_R^{\frac{1}{p_i'}} 
\lesssim 1. 
\end{equation}
Similarly, we can show that
\begin{equation}\label{eq:eq30}
\langle w^p\rangle_{R}^{\frac 1p}\prod_{i=1}^n \langle \sigma_i\rangle_{R^+}^{\frac{1}{p_i'}} 
\lesssim 1. 
\end{equation}

By H\"older's inequality we have that 
\[
\langle w^{-\frac p{np-1}}\rangle_{R^+}^{\frac{np-1}p}\le \prod_{i=1}^n \langle \sigma_i\rangle_{R^+}^{\frac 1{p_i'}}.
\]
Hence, \eqref{eq:eq30} shows that
\[
\langle w^p\rangle_{R}^{\frac 1p} \langle w^{-\frac p{np-1}}\rangle_{R^+}^{\frac{np-1}p}
\lesssim 1.
\]
Therefore,
\begin{align*}
\frac{\langle w^p\rangle_{R}^{\frac 1p}}{\langle w^p\rangle_{R^+}^{\frac 1p}}
=\frac{\langle w^p\rangle_{R}^{\frac 1p}\langle w^{-\frac p{np-1}}\rangle_{R^+}^{\frac{np-1}p}}{\langle w^p\rangle_{R^+}^{\frac 1p}\langle w^{-\frac p{np-1}}\rangle_{R^+}^{\frac{np-1}p}} 
\lesssim 1,
\end{align*}
where the denominator in the middle term was $\ge 1$. 
Thus, $\langle w^p\rangle_{R}^{\frac 1p} \lesssim \langle w^p\rangle_{R^+}^{\frac 1p}$, which together with \eqref{eq:eq29} gives that
$
\langle w^p\rangle_{R}^{\frac 1p}\prod_{i=1}^n \langle \sigma_i\rangle_R^{\frac 1{p_i'}}  
\lesssim 1.
$

\section{Maximal functions}
It was proved in \cite{GLTP} that the multilinear bi-parameter (or multi-parameter) maximal function is bounded
with respect to the genuinely multilinear bi-parameter weights. We give a new efficient proof of this. Let $\calD = \calD^1 \times \calD^2$
be a fixed lattice of dyadic rectangles and define
$$
M_{\calD}(f_1, \ldots, f_n) = \sup_{R \in \calD} \prod_{i=1}^n \ave{ |f_i| }_R 1_R.
$$
\begin{prop}\label{prop:prop2}
If $1 < p_1, \ldots, p_n \le \infty$ and $1/p = \sum_{i=1}^n 1/p_i$ we have
$$
\|M_{\calD}(f_1, \ldots, f_n)w \|_{L^p} \lesssim \prod_{i=1}^n \|f_i w_i\|_{L^{p_i}} 
$$
for all multilinear bi-parameter weights $\vec w \in A_{\vec p}$. 
\end{prop}
\begin{proof}
Our proof is based on the proof of the case $\vec p =  (p_1, \ldots, p_n) = (\infty, \ldots, \infty)$ and extrapolation, Theorem \ref{thm:ext}.
We have
$$
\sup_R \Big[\prod_i \langle w_i^{-1} \rangle_R \Big]\cdot \esssup_R w = [\vec{w}]_{A_{\vec p}},
$$
and therefore
$$
\prod_i \langle w_i^{-1} \rangle_R \lesssim \frac{1}{\esssup_R w}.
$$
For every $R \in \calD$ let $N_R \subset R$ be such that $|N_R| = 0$ and $w(x) \le \esssup_R w$ for all $x \in R \setminus N_R$.
Let $N = \bigcup_{R \in \calD} N_R$. Then $|N| = 0$ and for every $x \in \R^d \setminus N$ we have
$$
\frac{1}{w(x)} \ge \sup_{R \in \calD} \frac{1_R(x)}{\esssup_R w}.
$$
Thus, we have
\begin{align*}
M_{\calD}(f_1, \ldots, f_n)(x)w(x) &\le \Big[\prod_i \|f_i w_i\|_{L^{\infty}} \Big]\sup_{R \in \calD} \Big[ 1_R(x) \prod_i \ave{ w_i^{-1}}_R \Big] \cdot w(x) \\
&\lesssim \Big[\prod_i \|f_i w_i\|_{L^{\infty}}\Big] \sup_{R \in \calD} \Big[ \frac{1_R(x) }{\esssup_R w} \Big] \cdot w(x) \le \prod_i \|f_i w_i\|_{L^{\infty}} 
\end{align*}
almost everywhere, and so $\|M_{\calD}(f_1, \ldots, f_n)w\|_{L^{\infty}} \lesssim \prod_i \|f_i w_i\|_{L^{\infty}}$
as desired.

\end{proof}

If an average is with respect to a different measure $\mu$ than the Lebesgue measure, we write
$\langle f \rangle_R^{\mu} := \frac{1}{\mu(R)} \int_R f\ud \mu$ and define
$$
M_{\calD}^{\mu} f = \sup_R 1_R  \langle |f| \rangle_R^{\mu}.
$$
The following is a result of R. Fefferman \cite{RF3}. Recently, we also recorded a proof in \cite[Appendix B]{LMV:Bloom}.
\begin{prop}\label{prop:prop1}
Let $\lambda \in A_p$, $p \in (1,\infty)$, be a bi-parameter weight. Then for all $s \in (1,\infty)$ we have
$$
\| M_{\calD}^{\lambda} f \|_{L^s(\lambda)} \lesssim [\lambda]_{A_p}^{1+1/s} \|f\|_{L^s(\lambda)}.
$$
\end{prop}

We formulate some vector-valued versions of Proposition \ref{prop:prop1}. 
We state the following version
with two sequence spaces -- of course, a version with arbitrarily many also works.
Proposition \ref{prop:vecvalmax} is proved in the end of Section \ref{app:app1}.

\begin{prop}\label{prop:vecvalmax}
Let $\mu\in A_\infty$, $w\in A_p(\mu)$ and $1<p,s,t<\infty$. Then we have
\[
\left\| \Big\|\big\|\{M^\mu f_j^i\}\big\|_{\ell^s}\Big\|_{\ell^t}\right\|_{L^p(w\mu)}\lesssim \left\|\Big\|\big\|\{f_j^i\}\big\|_{\ell^s}\Big\|_{\ell^t}\right\|_{L^p(w\mu)}.
\]In particular, we have 
\[
\left\| \Big\|\big\|\{M^\mu f_j^i\}\big\|_{\ell^s}\Big\|_{\ell^t}\right\|_{L^p(\mu)}\lesssim \left\|\Big\|\big\|\{f_j^i\}\big\|_{\ell^s}\Big\|_{\ell^t}\right\|_{L^p(\mu)}.
\]
\end{prop}

Finally, we point out that everything in this section works easily in the general multi-parameter situation.

\section{Square functions}\label{sec:SquareFunctions}
Let $\calD = \calD^1 \times \calD^2$
be a fixed lattice of dyadic rectangles. We define the square functions
$$
S_{\calD} f = \Big( \sum_{R \in \calD}  |\Delta_R f|^2 \Big)^{1/2}, \,\, S_{\calD^1}^1 f =  \Big( \sum_{I^1 \in \calD^1}  |\Delta_{I^1}^1 f|^2 \Big)^{1/2}
$$
and define $S_{\calD^2}^2 f$ analogously.

The following lower square function estimate valid for $A_{\infty}$ weights is important for us. The importance
comes from the fact that by Lemma \ref{lem:lem1} some of the key weights $w^p$ and $w_i^{-p_i'}$ are
at least $A_{\infty}$ for the multilinear weights of Definition \ref{defn:defn1}. 
\begin{lem}\label{lem:lem3} There holds
$$
\|f\|_{L^p(w)} \lesssim \|S_{\calD^i}^i f\|_{L^p(w)} \lesssim  \|S_{\calD} f\|_{L^p(w)}
$$
for all $p \in (0, \infty)$ and bi-parameter weights $w \in A_{\infty}$.
\end{lem}
For a proof of the one-parameter estimate see \cite[Theorem 2.5]{Wi}. The bi-parameter results can be deduced using the following extremely useful $A_{\infty}$ extrapolation result \cite{CUMP}, which will be applied several times during the paper. We also mention that  square function estimates related to Lemma \ref{lem:lem3} also appear in \cite{BM3}.

\begin{lem}\label{lem:lem4}
Let $(f,g)$ be a pair of non-negative functions. Suppose that there exists some $0<p_0<\infty$ such that for every $w\in A_\infty$ we have
$$
\int f^{p_0} w \lesssim \int g^{p_0} w.
$$
Then for all $0<p<\infty$ and $w\in A_\infty$ we have
$$
\int f^{p} w \lesssim \int g^{p} w.
$$
\end{lem}

\begin{proof}[Proof of Lemma \ref{lem:lem3}]
Let $w \in A_\infty$ be a bi-parameter weight. 
The first estimate in the statement follows from the one-parameter result \cite[Theorem 2.5]{Wi} and
the fact that  $w(x_1, \cdot) \in A_\infty(\R^{d_2})$ and $w(\cdot, x_2) \in A_\infty(\R^{d_1})$.
Using this, we have that
$$
\| f \|_{L^2(w)} \lesssim \| S^1_{\calD^1}f\|_{L^2(w)}
= \Big(\sum_{I^1 \in \calD^1} \| \Delta^1_{I^1}f\|_{L^2(w)}^2 \Big)^{\frac 12}.
$$
For each $I^1$ we again use the one-parameter estimate to get
$$
\| \Delta^1_{I^1}f\|_{L^2(w)}
\lesssim \|S^2_{\calD^2} \Delta^1_{I^1}f\|_{L^2(w)}
=\Big( \sum_{I^2 \in \calD^2}\|\Delta^2_{I^2} \Delta^1_{I^1}f\|_{L^2(w)}^2 \Big)^{\frac 12}.
$$
Since $\Delta^2_{I^2} \Delta^1_{I^1}f= \Delta_{I^1 \times I^2} f$, inserting the last estimate into the previous one shows
that
$$
\| f \|_{L^2(w)} \lesssim \Big( \sum_{I^1 \times I^2 \in \calD^1 \times \calD^2} 
\|\Delta_{I^1 \times I^2} f\|_{L^2(w)}^2 \Big)^{\frac 12}
= \| S_{\calD} f \|_{L^2(w)}.
$$
Since this holds for every bi-parameter weight $w \in A_\infty$, Lemma \ref{lem:lem4} concludes the proof.
We point out that with further extrapolation we could obtain vector-valued versions analogous to Proposition 
\ref{prop:vecvalmax}, see the end of Section \ref{app:app1}.
\end{proof}

\begin{rem}\label{rem:rem1}
We often use the lower square function estimate with the additional observation that we e.g. have for all $k = (k_1, k_2) \in \{0,1,\ldots\}^2$ that
$$
S_{\calD} f = \Big( \sum_{K = K^1 \times K^2 \in \calD}  |\Delta_{K,k} f|^2 \Big)^{1/2}, \qquad \Delta_{K,k} = \Delta_{K^1,k_1}^1 \Delta_{K^2, k_2}^2.
$$
This simply follows from disjointness.
\end{rem}

For $k= (k_1, k_2)$ we define the following family of $n$-linear square functions. First, we set
$$
A_1(f_1, \ldots, f_n) = A_{1,k}(f_1, \ldots, f_n)
= \Big( \sum_{K \in \calD}  \langle | \Delta_{K,k} f_1 | \rangle_K ^2 \prod_{j=2}^n \langle |f_j| \rangle_K^2 1_K \Big)^{\frac{1}{2}}.
$$
In addition, we understand this so that $A_{1,k}$ can also take any one of the symmetric forms, where each $\Delta_{K^i, k_i}^i$ appearing in
$\Delta_{K,k} = \Delta_{K^1,k_1}^1 \Delta_{K^2, k_2}^2$ can alternatively be associated with any of the other functions $f_2, \ldots, f_n$. That is,
$A_{1,k}$ can, for example, also take the form
$$
A_{1,k}(f_1, \dots, f_n) = 
\Big( \sum_{K \in \calD} \langle | \Delta^2_{K^2,k_2} f_1 | \rangle_K^2
\langle | \Delta^1_{K^1,k_1} f_2| \rangle_K^2 \prod_{j=3}^{n} \langle |f_j| \rangle_K^2 1_K \Big)^{\frac 12}.
$$
For $k = (k_1, k_2, k_3)$ we define
\begin{equation}\label{eq:eq11}
\begin{split}
&A_{2,k}(f_1, \ldots, f_n) \\
&= \Big( \sum_{K^2 \in \calD^2} \Big( \sum_{K^1 \in \calD^1}
\langle|\Delta^2_{K^2, k_1}f_1|\rangle_{K}\langle|\Delta^1_{K^1, k_2}f_2|\rangle_{K}
\langle|\Delta^1_{K^1, k_3}f_3|\rangle_{K} \prod_{j=4}^{n} \langle |f_j|\rangle_K1_{K} 
\Big)^2\Big)^{ \frac 12},
\end{split}
\end{equation}
where we again understand this as a family of square functions. First, the appearing three martingale blocks can be 
associated with different functions, too. Second, we can have the $K^1$ summation out and the $K^2$ summation in (we can interchange them), but then
we have two martingale blocks with $K^2$ and one martingale block with $K^1$.

Finally, for $k = (k_1, k_2, k_3, k_4)$ we define
$$
A_{3,k}(f_1, \ldots, f_n) =  \sum_{K \in \calD} \langle | \Delta_{K,(k_1, k_2)} f_1| \rangle_K
\langle | \Delta_{K,(k_3, k_4)} f_2| \rangle_K \prod_{j=3}^n \langle |f_j| \rangle_K 1_K,
$$
where this is a family with two martingale blocks in each parameter, which can be moved around.
\begin{thm}\label{thm:thm3}
If $1 < p_1, \ldots, p_n \le \infty$ and $\frac{1}{p} = \sum_{i=1}^n \frac{1}{p_i}> 0$ we have
$$
\|A_{j,k}(f_1, \ldots, f_n)w \|_{L^p} \lesssim \prod_{i=1}^n \|f_i w_i\|_{L^{p_i}}, \quad j=1,2,3,
$$
for all multilinear bi-parameter weights $\vec w \in A_{\vec p}$.
\end{thm}

\begin{proof}
The proofs of all of the cases have the same underlying idea based on an iterative use of duality and the lower square function estimate
until all of the cancellation has been utilised. One can also realise that the result for $A_{3,k}$ follows using the above
scheme just once if the result is first proved for $A_{1,k}$ and $A_{2,k}$.

We show the proof for $A_{2,k}$ with the explicit form \eqref{eq:eq11}.
Fix some $\vec p = (p_1, \ldots, p_n)$ with $1 < p_i < \infty$ and $p > 1$. This is enough by extrapolation, Theorem \ref{thm:ext}.
To estimate
$\|A_{2,k}(f_1, \ldots, f_n)w \|_{L^p}$ we take a sequence $(f_{n+1,K^2})_{K^2} \subset L^{p'}(\ell^2)$ with a norm $\| (f_{n+1,K^2})_{K^2} \|_{L^{p'}(\ell^2)} \le 1$
and look at
\begin{equation}\label{eq:A1Dual}
\sum_{K}\langle|\Delta^2_{K^2, k_1}f_1|,1_K\rangle\langle|\Delta^1_{K^1, k_2}f_2|\rangle_{K}
\langle|\Delta^1_{K^1, k_3}f_3|\rangle_{K} \prod_{j=4}^{n} \langle |f_j|\rangle_K \langle f_{n+1,K^2}w \rangle_{K}.
\end{equation}
There holds that
\begin{equation}\label{eq:RemAbs}
\langle|\Delta^2_{K^2, k_1}f_1|,1_K\rangle = \langle  \Delta^2_{K^2, k_1}f_1  ,  \varphi_{K^2, f_1} \rangle 
= \langle   f_1  , \Delta^2_{K^2, k_1} \varphi_{K^2, f_1} \rangle, \qquad |\varphi_{K^2, f_1}| \le 1_K.
\end{equation}
We now get that \eqref{eq:A1Dual} is less than $\| f_1 w_1\|_{L^{p_1}}$ multiplied by
\begin{equation*}
\Big \| \sum_{K} \langle f_{n+1,K^2}w \rangle_{K} \langle|\Delta^1_{K^1, k_2}f_2|\rangle_{K}
\langle|\Delta^1_{K^1, k_3}f_3|\rangle_{K} \prod_{j=4}^{n} \langle |f_j|\rangle_K  \Delta^2_{K^2, k_1} \varphi_{K^2, f_1} w_1^{-1}\Big \|_{L^{p_1'}}.
\end{equation*}

We will now apply the lower square function estimate $\|gw_1^{-1}\|_{L^{p_1'}} \lesssim \|S_{\calD^2}^2(g) w_1^{-1} \|_{L^{p_1'}}$, Lemma \ref{lem:lem3},
with the weight $w_1^{-p_1'} \in A_{\infty}$ (see Lemma \ref{lem:lem1}). Here we use the block form of Remark \ref{rem:rem1}.
Using also that $|\Delta^2_{K^2, k_1} \varphi_{K^2, f_1}| \lesssim 1_K$
we get that the last norm is dominated by
\begin{equation*}
\Big \|\Big( \sum_{K^2} \Big( \sum_{K^1} \langle |f_{n+1,K^2}|w \rangle_{K} \langle|\Delta^1_{K^1, k_2}f_2|\rangle_{K}
\langle|\Delta^1_{K^1, k_3}f_3|\rangle_{K} \prod_{j=4}^{n} \langle |f_j|\rangle_K  
1_K \Big)^2 \Big)^{\frac 12} w_1^{-1}\Big \|_{L^{p_1'}}.
\end{equation*}
We still have cancellation to use in the form of the other two martingale differences and will continue the process.

We repeat the argument from above -- this gives that the previous term is dominated by $\| f_2 w_2 \|_{L^{p_2}}$ multiplied by
$$
\Big \|\Big( \sum_{K^1} \Big( \sum_{K^2} \langle |f_{n+1,K^2}|w \rangle_{K} \langle |f_{1,K^2}|w_1^{-1}\rangle_K
\langle|\Delta^1_{K^1, k_3}f_3|\rangle_{K} \prod_{j=4}^{n} \langle |f_j|\rangle_K  
1_K \Big)^2 \Big)^{\frac 12} w_2^{-1}\Big \|_{L^{p_2'}}
$$
where $\| (f_{1,K^2})_{K^2} \|_{L^{p_1}(\ell^2)} \le 1$. Running this argument one more time finally gives us
that this is dominated by $\| f_3w_3\|_{L^{p_3}}$ multiplied by
\begin{equation*}
\begin{split}
\Big \| \Big(\sum_{K^1} & \Big( \sum_{K^2} \langle |f_{n+1,K^2}|w \rangle_{K}  \langle |f_{1,K^2}|w_1^{-1}\rangle_K \langle |f_{2, K^1}|w_2^{-1}\rangle_K 
\prod_{j=4}^{n} \langle |f_j|\rangle_K 1_K \Big)^2 \Big)^{\frac 12} w_3^{-1}\Big\|_{L^{p_3'}} \\
& \le \Big \| \Big(\sum_{K^1} \Big( \sum_{K^2} M_\calD( f_{n+1,K^2}w, f_{1,K^2}w_1^{-1}, f_{2, K^1}w_2^{-1},  f_4, \dots, f_n) \Big)^2 \Big)^{\frac 12} w_3^{-1}\Big\|_{L^{p_3'}},
\end{split}
\end{equation*}
where $\| (f_{2,K^1})_{K^1} \|_{L^{p_2}(\ell^2)} \le 1$.

Using Lemma  \ref{lem:lem7} three times (we dualized three times) shows that
$$
(w^{-1}, w_1,w_2,w_4, \dots, w_n) \in A_{(p',p_1,p_2,p_4, \dots, p_n)}.
$$
The maximal function satisfies the weighted
$$
L^{p'}(\ell^\infty_{K^1}(\ell^2_{K^2})) \times L^{p_1}(\ell^\infty_{K^1}(\ell^2_{K^2})) \times L^{p_2}(\ell^2_{K^1}(\ell^\infty_{K^2})) 
\times L^{p_4} \times \dots \times L^{p_n} \to L^{p_3'}(\ell^2_{K^1}(\ell^1_{K^2}))
$$
estimate. This gives that the last norm above is dominated by
\begin{equation*}
\| ( f_{n+1,K^2}ww^{-1})_{K^2} \|_{L^{p'}(\ell^2)}
\| ( f_{1,K^2}w_1^{-1}w_1)_{K^2} \|_{L^{p_1}(\ell^2)}
\| ( f_{2, K^1}w_2^{-1}w_2)_{K^1} \|_{L^{p_2}(\ell^2)}
\prod_{i=4}^n \| f_iw_i \|_{L^{p_i}},
\end{equation*}
where the first three norms are $\le 1$. This concludes the proof for $A_{2,k}$ and the rest of the cases are similar.
\end{proof}
We also record some linear estimates. We will need these when we deal
with the most complicated model operators -- the partial paraproducts.
\begin{prop}\label{prop:prop3}
For $u \in A_{\infty}$ and $p, s \in (1, \infty)$ we have
$$
\Big\| \Big[ \sum_m \Big( \sum_{K \in \calD} \langle |\Delta_{K,k} f_m| \rangle_K^2 \frac{1_K}{\langle u \rangle_K^2} \Big)^{\frac{s}{2}} \Big]^{\frac{1}{s}} u^{\frac{1}{p}}
\Big\|_{L^p}
\lesssim \Big\| \Big( \sum_m |f_m|^s \Big)^{\frac{1}{s}} u^{-\frac{1}{p'}} \Big\|_{L^p}.
$$
\end{prop}
\begin{proof}
By \eqref{eq:eq8} we have for all $n \ge 2$ that
$$
1 \le \langle u \rangle_K \Big\langle u^{-\frac{1}{n-1}} \Big\rangle_K^{n-1}.
$$
Simply using this we reduce to
\begin{align*}
\Big\| \Big[& \sum_m \Big( \sum_{K \in \calD} \langle |\Delta_{K,k} f_m| \rangle_K^2 \Big\langle u^{-\frac{1}{n-1}} \Big\rangle_K^{2(n-1)}1_K\Big)^{\frac{s}{2}} \Big]^{\frac{1}{s}} u^{\frac{1}{p}}
\Big\|_{L^p} \\
&= \Big\| \Big[ \sum_m A_{1,k}\big(f_m, u^{-\frac{1}{n-1}}, \ldots, u^{-\frac{1}{n-1}}\big)^s  \Big]^{\frac{1}{s}} u^{\frac{1}{p}}
\Big\|_{L^p},
\end{align*}
where $A_{1,k}$ is a suitable square function as in Theorem \ref{thm:thm3}.

We then fix $n$ large enough so that $u \in A_{n}$.
We then notice that this implies that
\begin{equation}\label{eq:InftynLin}
 \big(u^{-\frac{1}{p'}}, u^{\frac{1}{n-1}}, \ldots, u^{\frac{1}{n-1}}\big) \in A_{(p, \infty, \ldots, \infty)}.
 \end{equation}
To see this, notice that the target weight associated with this tuple is $u^{-\frac{1}{p'}} u = u^{\frac{1}{p}}$ and that the target exponent is $p$, and so
$$
\big[\big(u^{-\frac{1}{p'}}, u^{\frac{1}{n-1}}, \ldots, u^{\frac{1}{n-1}}\big)\big]_{A_{(p, \infty, \ldots, \infty)}} =
\sup_R \langle u \rangle_R^{1/p} \langle u \rangle_R^{1/p'} \big\langle  u^{-\frac{1}{n-1}} \big\rangle_R^{n-1} = [u]_{A_n} < \infty.
$$

It remains to use the weighted (with the weight \eqref{eq:InftynLin}) vector-valued estimate $L^p(\ell^s) \times L^{\infty}\times \cdots \times L^{\infty} \to L^p(\ell^s)$ of $A_{1,k}$, which follows by Theorem \ref{thm:thm3} and \eqref{extrapol:vv*}. 
\end{proof}
\begin{rem}
It is possible to prove the above proposition also directly with the duality and lower square function strategy that was used in the proof of Theorem \ref{thm:thm3}.
\end{rem}


\section{Dyadic model operators}\label{sec:dmo}
In this section we are working with a fixed set of dyadic rectangles $\calD = \calD^1 \times \calD^2$. All the model operators depend on this lattice,
but it is not emphasised in the notation. 
\subsection{Shifts}
Let $k=(k_1, \dots, k_{n+1})$, where $k_j = (k_j^1, k_j^2) \in \{0,1,\ldots\}^2$. 
An $n$-linear bi-parameter shift $S_k$ takes the form
\begin{equation*}\label{eq:S2par}
\langle S_k(f_1, \ldots, f_n), f_{n+1}\rangle = \sum_{K} \sum_{\substack{R_1, \ldots, R_{n+1} \\ R_j^{(k_j)} = K }}
a_{K, (R_j)} \prod_{j=1}^{n+1} \langle f_j, \wt h_{R_j} \rangle.
\end{equation*}
Here $K, R_1, \ldots, R_{n+1} \in \calD = \calD^1 \times \calD^2$, $R_j = I_j^1 \times I_j^2$, $R_j^{(k_j)} := (I_j^1)^{(k_j^1)} \times (I_j^2)^{(k_j^2)}$ and 
$\wt h_{R_j} = \wt h_{I_j^1} \otimes \wt h_{I_j^2}$. Here we assume that for $m \in \{1,2\}$
there exist two indices $j^m_0,j_1^m \in \{1, \ldots, n+1\}$, $j^m_0 \not =j^m_1$, so that $\wt h_{I_{j^m_0}^m}=h_{I_{j^m_0}^m}$, $\wt h_{I_{j^m_1}^m}=h_{I_{j^m_1}^m}$ and for the remaining indices $j \not \in \{j^m_0, j^m_1\}$ we have $\wt h_{I_j^m} \in \{h_{I_j^m}^0, h_{I_j^m}\}$.
Moreover, $a_{K,(R_j)} = a_{K, R_1, \ldots ,R_{n+1}}$ is a scalar satisfying the normalization
\begin{equation}\label{eq:Snorm2par}
|a_{K,(R_j)}| \le \frac{\prod_{j=1}^{n+1} |R_j|^{1/2}}{|K|^{n}}.
\end{equation}

\begin{thm}
Suppose $S_k$ is an $n$-linear bi-parameter shift, $1 < p_1, \ldots, p_n \le \infty$ and $\frac{1}{p} = \sum_{i=1}^n \frac{1}{p_i}> 0$. Then we have
$$
\|S_k(f_1, \ldots, f_n)w \|_{L^p} \lesssim \prod_{i=1}^n \|f_i w_i\|_{L^{p_i}} 
$$
for all multilinear bi-parameter weights $\vec w \in A_{\vec p}$. The implicit constant does not depend on $k$.
\end{thm}

\begin{proof}
We use duality to always reduce to one of the operators of type $A_{3}$ in Theorem \ref{thm:thm3}.
Performing the proof like this has the advantage that the form of the shift really plays no role -- it just affects which type of $A_3$ operator we get.
For example, we consider the explicit case
$$
S_k(f_1, \dots, f_n)
= \sum_{K} A_K(f_1, \ldots, f_n),
$$
where
$$
A_K(f_1, \ldots, f_n) = 
\sum_{\substack{R_1, \ldots, R_{n+1} \\ R_j^{(k_j)} = K }}
a_{K, (R_j)} \langle f_1, h_{R_1} \rangle \prod_{j=2}^{n} \langle f_j, \wt h_{R_j} \rangle h_{R_{n+1}}.
$$
Fix some $\vec p = (p_1, \ldots, p_n)$ with $1 < p_i < \infty$ and $p > 1$, which is enough by extrapolation. We will dualise using $f_{n+1}$ with $\|f_{n+1}w^{-1}\|_{L^{p'}} \le 1$.
The normalisation of the shift coefficients gives the 
direct estimate
\begin{equation*}
\begin{split}
\sum_{K}& |\langle A_K(f_1, \ldots, f_n), f_{n+1} \rangle | \\
& \le \sum_K \sum_{\substack{R_1, \ldots, R_{n+1} \\ R_j^{(k_j)} = K }} 
\frac{\prod_{j=1}^{n+1} |R_j|^{1/2}}{|K|^{n}} 
\Big|\langle \Delta_{K,k_1} f_1, h_{R_1} \rangle \prod_{j=2}^{n} \langle f_j, \wt h_{R_j} \rangle 
\langle \Delta_{K,k_{n+1}}f_{n+1},h_{R_{n+1}}\rangle\Big| \\
& \le \sum_K \sum_{\substack{R_1, \ldots, R_{n+1} \\ R_j^{(k_j)} = K }} 
\frac{1}{|K|^{n}} 
\langle |\Delta_{K,k_1} f_1|, 1_{R_1} \rangle \prod_{j=2}^{n} \langle |f_j|, 1_{R_j} \rangle 
\langle |\Delta_{K,k_{n+1}}f_{n+1}|,1_{R_{n+1}}\rangle \\ 
&\le  \sum_K \langle | \Delta_{K,k_1} f_1 | \rangle_K 
\prod_{j=2}^{n} \langle |f_j| \rangle_K\langle |\Delta_{K,k_{n+1}}  f_{n+1}| \rangle_K |K| \\
& = \Big\| \sum_K \langle | \Delta_{K,k_1} f_1 | \rangle_K 
\prod_{j=2}^{n} \langle |f_j| \rangle_K\langle |\Delta_{K,k_{n+1}}  f_{n+1}| \rangle_K 1_K \Big\|_{L^1},
\end{split}
\end{equation*}
where we used \eqref{eq:HaarMart} in the first step in the passage from Haar functions into martingale differences.
Notice that
\begin{equation}\label{eq:eq12}
(w_1, \cdots, w_n, w^{-1})\in A_{(p_1,\cdots, p_n, p')}, \qquad w=\prod_{i=1}^n w_i.
\end{equation}
The target weight associated to this data is $ww^{-1} = 1$ and the target exponent is $1/p + 1/p' = 1$.
By using Theorem \ref{thm:thm3} with a suitable $A_{3}(f_1, \ldots, f_{n+1})$ and the above weight we can directly dominate this by
$$
\Big[\prod_{i=1}^n \|f_i w_i\|_{L^{p_i}}  \Big]\cdot \|f_{n+1} w^{-1}\|_{L^{p'}} \le \prod_{i=1}^n \|f_i w_i\|_{L^{p_i}}.
$$
We are done.
\end{proof}

\subsection{Partial paraproducts}
Let $k=(k_1, \dots, k_{n+1})$, where $k_j \in \{0,1,\ldots\}$.
An $n$-linear bi-parameter partial paraproduct $(S\pi)_k$ with the paraproduct component on $\R^{d_2}$ takes the form
\begin{equation}\label{eq:Spi}
\langle (S\pi)_k(f_1, \ldots, f_n), f_{n+1} \rangle = 
\sum_{K = K^1 \times K^2} \sum_{\substack{ I^1_1, \ldots, I_{n+1}^1 \\ (I_j^1)^{(k_j)} = K^1}} a_{K, (I_j^1)} \prod_{j=1}^{n+1} \langle f_j, \wt h_{I_j^1} \otimes u_{j, K^2} \rangle,
\end{equation}
where the functions $\wt h_{I_j^1}$ and $u_{j, K^2}$ satisfy the following.
There are $j_0,j_1 \in \{1, \ldots, n+1\}$, $j_0 \not =j_1$, so that $\wt h_{I_{j_0}^1}=h_{I_{j_0}^1}$, $\wt h_{I_{j_1}^1}=h_{I_{j_1}^1}$ and for the remaining indices $j \not \in \{j_0, j_1\}$ we have $\wt h_{I_j^1} \in \{h_{I_j^1}^0, h_{I_j^1}\}$. There is $j_2 \in \{1, \ldots, n+1\}$ so that $u_{j_2, K^2} = h_{K^2}$ and for the remaining indices $j \ne j_2$ we have
$u_{j, K^2} = \frac{1_{K^2}}{|K^2|}$.
Moreover, the coefficients are assumed to satisfy
\begin{equation}\label{eq:PPNorma}
\| (a_{K, (I_j^1)})_{K^2} \|_{\BMO} = \sup_{K^2_0 \in \calD^2} \Big( \frac{1}{|K^2_0|} \sum_{K^2 \subset K^2_0} |a_{K, (I_j^1)}|^2 \Big)^{1/2}  
\le \frac{\prod_{j=1}^{n+1} |I_j^1|^{\frac 12}}{|K^1|^{n}}.
\end{equation}
Of course, $(\pi S)_k$ is defined symmetrically.

The following $H^1$-$\BMO$ duality type estimate is well-known and elementary:
\begin{equation}\label{eq:H1BMO}
\sum_{K^2} |a_{K^2}| |b_{K^2}| \lesssim \| (a_{K^2} ) \|_{\BMO} \Big\| \Big( \sum_{K^2} |b_{K^2}|^2 \frac{1_{K^2}}{|K^2|} \Big)^{1/2} \Big \|_{L^1}.
\end{equation}
Such estimates have natural multi-parameter analogues also, and the proofs in all parameters are analogous. See e.g. \cite[Equation (4.1)]{MO}.

Our result for the partial paraproducts has a significantly more difficult proof than for the other model operators. It is also more inefficient
in that is produces an exponential -- although crucially with an arbitrarily small exponent -- dependence on the complexity. This has
some significance for the required kernel regularity of CZOs, but a standard $t \mapsto t^{\alpha}$ type continuity modulus will still suffice.
\begin{thm}\label{thm:thm4}
Suppose $(S\pi)_k$ is an $n$-linear partial paraproduct, $1 < p_1, \ldots, p_n \le \infty$ and $\frac{1}{p} = \sum_{i=1}^n \frac{1}{p_i}> 0$. Then, for every 
$0<\beta \le 1$ we have
$$
\|(S\pi)_k(f_1, \ldots, f_n)w \|_{L^p} \lesssim_\beta 2^{\max_j k_j \beta}\prod_{i=1}^n \|f_i w_i\|_{L^{p_i}} 
$$
for all multilinear bi-parameter weights $\vec w \in A_{\vec p}$.
\end{thm}
\begin{proof}
Recall that $(S\pi)_k$ is of the form \eqref{eq:Spi}. Recall also the indices $j_0$ and $j_1$, which say that $\wt h_{I^1_j}=h_{I^1_j}$ at least for 
$j\in \{j_0, j_1\}$, and the index $j_2$, which specifies the place of $h_{K^2}$ in the second parameter.
It makes no difference for the argument what the indices $j_0$ and $j_1$ are, so we assume that $j_0=1$ and $j_1=2$. It makes a small difference 
whether $j_2 \in \{j_0,j_1\}$ or $j_2 \not \in \{j_0,j_1\}$, so we do not specify $j_2$ yet. To make the following formulae shorter we write $\wt h_{I^1_j}$ 
for every $j$ but keep in mind that these are cancellative at least for $j \in \{1,2\}$. We define
$$
A_{K^2}(g_1, \dots, g_{n+1})
= \prod_{j=1}^{n+1} \langle g_j, u_{j,K^2} \rangle
$$ 
and write $(S\pi)_k$ in the form
$$
\langle (S\pi)_k(f_1, \ldots, f_n), f_{n+1} \rangle = 
\sum_{K = K^1 \times K^2} \sum_{\substack{ I^1_1, \ldots, I_{n+1}^1 \\ (I_j^1)^{(k_j)} = K^1}} a_{K, (I_j^1)} 
A_{K^2}(\langle f_{1}, \wt h_{I_{1}^1}\rangle_1, \dots, \langle f_{n+1}, \wt h_{I_{n+1}^1}\rangle_1 ).
$$

Fix some $\vec p = (p_1, \ldots, p_n)$ with $1 < p_i < \infty$ and $p > 1$, which is enough by extrapolation. We will dualise using $f_{n+1}$ with $\|f_{n+1}w^{-1}\|_{L^{p'}} \le 1$. We may assume $f_j \in L^{\infty}_c$.
The $H^1$-$\BMO$ duality \eqref{eq:H1BMO} gives that
\begin{equation}\label{eq:eq13}
\begin{split}
|\langle (S\pi)_k(f_1, \ldots, f_n), f_{n+1} \rangle|
&\lesssim \sum_{K^1} \sum_{\substack{ I^1_1, \ldots, I_{n+1}^1 \\ (I_j^1)^{(k_j)} = K^1}}\Bigg[ \frac{\prod_{j=1}^{n+1}|I^1_j|^{\frac 12}}{|K^1|^n} \\
&\int_{\R^{d_2}} \Big( \sum_{K^2} |A_{K^2}(\langle f_{1}, \wt h_{I_{1}^1}\rangle_1, \dots, \langle f_{n+1}, \wt h_{I_{n+1}^1}\rangle_1 )|^2 
\frac{1_{K^2}}{|K^2|} \Big)^{\frac{1}{2}}\Bigg].
\end{split}
\end{equation}

Suppose $j \in \{3, \dots, n+1\}$ is such that $\wt h_{I^1_j}=h_{I^1_j}^0$ and $k_j>0$, that is, we have non-cancellative Haar functions and non-zero complexity.  
We expand 
$$
|I^1_j|^{-\frac{1}{2}} \langle f_j, h^0_{I^1_j} \rangle_1
=\langle f\rangle_{I_j^1,1}
=\langle f_j \rangle_{K^1,1}+\sum_{i_j=1}^{k_j} \langle \Delta^1_{(I^1_j)^{(i_j)}} f_j\rangle_{(I_j^1)^{(i_j-1)},1}.
$$
For convenience, we further write that
$$
\langle \Delta^1_{(I^1_j)^{(i_j)}} f_j\rangle_{(I_j^1)^{(i_j-1)},1}
= \langle h_{(I^1_j)^{(i_j)}} \rangle_{(I^1_j)^{(i_j-1)}} \langle f_j, h_{(I^1_j)^{(i_j)}} \rangle_1, 
$$
where we are suppressing the summation over the $2^{d_1}-1$ different Haar functions. 
We perform these expansions inside the operators $A_{K^2}$, and take the sums out of the $\ell^2_{K^2}$ norm.
This gives that the right hand side of \eqref{eq:eq13} is less than a sum of at most $\prod_{j=3}^n(1+k_j)$ terms of the form 
\begin{equation*}
\begin{split}
\sum_{K^1}  \sum_{\substack{ I^1_1, \ldots, I_{n+1}^1 \\ (I_j^1)^{(k_j)} = K^1}}& \Bigg[ \frac{\prod_{j=1}^{n+1} |I^1_j| |(I^1_j)^{(i_j)}|^{-\frac 12}}{|K^1|^n} \\
&\int_{\R^{d_2}} \Big( \sum_{K^2} |A_{K^2}(\langle f_{1}, \wt h_{(I_{1}^1)^{(i_1)}}\rangle_1, \dots, \langle f_{n+1}, \wt h_{(I_{n+1}^1)^{(i_{n+1})}}\rangle_1 )|^2 
\frac{1_{K^2}}{|K^2|} \Big)^{\frac{1}{2}} \Bigg].
\end{split}
\end{equation*}
Here we have the following properties. 
If $j$ in an index such that we did not do the expansion related to $j$, then $i_j=0$. Thus, at least $i_1=i_2=0$. We also remind that
$\wt h_{(I_{j}^1)^{(i_j)}}=h_{(I_{j}^1)^{(i_j)}}$ for $j=1,2$.
If $i_j<k_j$, then 
$\wt h_{(I_{j}^1)^{(i_j)}}=h_{(I_{j}^1)^{(i_j)}}$. If $i_j=k_j$, then $\wt h_{(I_{j}^1)^{(i_j)}} \in \{h_{K^1}, h_{K^1}^0\}$. We can further rewrite this as
\begin{equation}\label{eq:eq14}
\sum_{K^1}  \sum_{\substack{ L^1_1, \ldots, L_{n+1}^1 \\ (L_j^1)^{(l_j)} = K^1}}  \frac{\prod_{j=1}^{n+1} |L^1_j|^{\frac 12} }{|K^1|^n} 
\int_{\R^{d_2}} \Big( \sum_{K^2} |A_{K^2}(\langle f_{1}, \wt h_{L_1}\rangle_1, \dots, \langle f_{n+1}, \wt h_{L_{n+1}}\rangle_1 )|^2 
\frac{1_{K^2}}{|K^2|} \Big)^{\frac{1}{2}}.
\end{equation}
This is otherwise analogous to the right hand side of \eqref{eq:eq13} except for the key difference that if a non-cancellative Haar function appears, then
the related complexity is zero.

We turn to estimate \eqref{eq:eq14}. We show that
\begin{equation}\label{eq:eq15}
\eqref{eq:eq14} \lesssim_\beta 2^{\max_j k_j \frac \beta 2 }\Big[\prod_{j=1}^n \|f_j w_j\|_{L^{p_j}}\Big]\| f_{n+1} w^{-1} \|_{L^{p'}}.
\end{equation}
Recalling that $\| f_{n+1} w^{-1} \|_{L^{p'}} \le 1$ this implies that the left hand side of \eqref{eq:eq13} satisfies
$$
LHS\eqref{eq:eq13} \lesssim_ \beta (1+\max_j k_j)^{n-1} 2^{\max_j k_j \frac \beta 2 } \prod_{j=1}^n \|f_j w_j\|_{L^{p_j}}
\lesssim_\beta 2^{\max_j k_j \beta  } \prod_{j=1}^n \|f_j w_j\|_{L^{p_j}},
$$
which proves the theorem.

Let $(v_1, \dots, v_{n+1}) \in A_{(2, \dots, 2)}$ and $v=\prod_{j=1}^{n+1} v_j$. We will prove the $(n+1)$-linear estimate
\begin{equation}\label{eq:eq16}
\begin{split}
\Bigg\| &
\sum_{K^1} \sum_{\substack{ L^1_1, \ldots, L_{n+1}^1 \\ (L_j^1)^{(l_j)} = K^1}} \Bigg[
\frac{\prod_{j=1}^{n+1} |L^1_j|^{\frac 12} }{|K^1|^n} \frac{1_{K^1}}{|K^1|} \\
&\Big( \sum_{K^2} |A_{K^2}(\langle f_{1},  \wt h_{L_1}\rangle_1, \dots, \langle f_{n+1}, \wt h_{L_{n+1}}\rangle_1 )|^2 
\frac{1_{K^2}}{|K^2|} \Big)^{\frac{1}{2}}\Bigg]
v\Bigg\|_{L^{\frac{2}{n+1}}} 
 \lesssim 2^{\max_j k_j \frac \beta 2 } \prod_{j=1}^{n+1} \| f_j v_j \|_{L^2}.
\end{split}
\end{equation}
Extrapolation, Theorem \ref{thm:ext}, then gives that
\begin{equation*}
\begin{split}
\Bigg\|
\sum_{K^1} \sum_{\substack{ L^1_1, \ldots, L_{n+1}^1 \\ (L_j^1)^{(l_j)} = K^1}} 
\frac{\prod_{j=1}^{n+1} |L^1_j|^{\frac 12} }{|K^1|^n}
\frac{1_{K^1}}{|K^1|}
\Big( \sum_{K^2} |A_{K^2}(\langle f_{1},  \wt h_{L_1}\rangle_1, \dots,& \langle f_{n+1}, \wt h_{L_{n+1}}\rangle_1 )|^2 
\frac{1_{K^2}}{|K^2|} \Big)^{\frac{1}{2}}
v\Bigg\|_{L^{q}} \\
& \lesssim 2^{\max_j k_j \frac \beta 2 }\prod_{j=1}^{n+1} \| f_j v_j \|_{L^{q_j}}
\end{split}
\end{equation*}
for all $q_1, \dots, q_{n+1} \in (1,\infty]$ such that $\frac{1}{q}=\sum_{j=1}^{n+1} \frac{1}{q_j}>0$ and for all $(v_1, \dots, v_{n+1}) \in A_{(q_1, \dots, q_{n+1})}$.
Applying this with the exponent tuple $(p_1, \dots, p_n, p')$ and the weight tuple $(w_1, \dots, w_n,w^{-1}) \in A_{(p_1, \dots, p_n,p')}$ gives \eqref{eq:eq15}.

It remains to prove \eqref{eq:eq16}. We denote $\sigma_j=v_j^{-2}$. The $A_{(2, \dots, 2)}$ condition gives that
$$
\langle v^{\frac{2}{n+1}} \rangle_K^{n+1} \prod_{j=1}^{n+1} \langle \sigma_j \rangle_K \lesssim 1.
$$
Using this we have
$$
|A_{K^2}(\langle f_{1},  \wt h_{L^1_1}\rangle_1, \dots, \langle f_{n+1}, \wt h_{L^1_{n+1}}\rangle_1 )|
\lesssim \frac{1}{\langle v^{\frac{2}{n+1}} \rangle_K^{n+1}} 
\Bigg|A_{K^2}\Bigg(\frac{\langle f_{1},  \wt h_{L^1_1}\rangle_1}{\langle \sigma_1\rangle_K}, \dots, 
\frac{\langle f_{n+1}, \wt h_{L^1_{n+1}}\rangle_1}{\langle \sigma_{n+1} \rangle_K} \Bigg)\Bigg|.
$$
For the moment we abbreviate the last $|A_{K^2}( \cdots)|$ as $c_{K,(L^1_j)}$. There holds that
\begin{equation*}
\begin{split}
\frac{1}{\langle v^{\frac{2}{n+1}} \rangle_K^{n+1}}c_{K,(L^1_j)}
&= \Bigg[ \frac{1}{\langle v^{\frac{2}{n+1}} \rangle_K} \Big\langle c_{K,(L^1_j)}^{\frac{1}{n+1}}1_K v^{-\frac{2}{n+1}}v^{\frac{2}{n+1}} \Big\rangle_K \Bigg ]^{n+1} \\
&\le \Big(M_\calD^{v^{\frac{2}{n+1}}}\Big(c_{K,(L^1_j)}^{\frac{1}{n+1}}1_K v^{-\frac{2}{n+1}}\Big)(x)\Big)^{n+1}
\end{split}
\end{equation*}
for all $x \in K$.

We substitute this into the left hand side of \eqref{eq:eq16}. This gives that the term there is dominated by
\begin{equation*}
\Bigg\| 
\sum_{K^1} \sum_{\substack{ L^1_1, \ldots, L_{n+1}^1 \\ (L_j^1)^{(l_j)} = K^1}}
\frac{\prod_{j=1}^{n+1} |L^1_j|^{\frac 12} }{|K^1|^{n+1}} 
\Big( \sum_{K^2} M_\calD^{v^{\frac{2}{n+1}}}\Big(c_{K,(L^1_j)}^{\frac{1}{n+1}}1_K v^{-\frac{2}{n+1}}\Big)^{2(n+1)}
\frac{1}{|K^2|} \Big)^{\frac{1}{2}}
v\Bigg\|_{L^{\frac{2}{n+1}}}.
\end{equation*}
We use the $L^2(\ell_{K^1,(L^1_j)}^{n+1}(\ell_{K^2}^{2(n+1)}))$ boundedness of the maximal function $M_\calD^{v^{\frac{2}{n+1}}}$, see Proposition \ref{prop:vecvalmax}.
This gives that the last norm is dominated by
\begin{equation}\label{eq:eq17}
\Bigg\| 
\sum_{K^1} \sum_{\substack{ L^1_1, \ldots, L_{n+1}^1 \\ (L_j^1)^{(l_j)} = K^1}}
\frac{\prod_{j=1}^{n+1} |L^1_j|^{\frac 12} }{|K^1|^{n+1}} 1_{K^1}
\Big( \sum_{K^2} c_{K,(L^1_j)}^{2}
\frac{1_{K^2}}{|K^2|} \Big)^{\frac{1}{2}}
v^{-1}\Bigg\|_{L^{\frac{2}{n+1}}}.
\end{equation}

Now, we recall what the numbers $c_{K,(L^1_j)}$ are. At this point it becomes relevant which of the Haar functions $\wt h_{L^1_j}$ are cancellative and 
what is the form of the operators $A_{K^2}$. We assume that $\wt h_{L^1_j}=h_{L^1_j}$ for $j=1, \dots, n$ 
and $\wt h_{L^1_{n+1}}=h_{L^1_{n+1}}^0=h_{K^1_{n+1}}^0$, which is a good representative of the general case. First, we assume that the index $j_2$, which specifies the place of $h_{K^2}$ in $A_{K^2}$, satisfies $j_2 \in \{1, \dots, n\}$.
The point is that then $\wt h_{L^1_{j_2}}=h_{L^1_{j_2}}$. For convenience of notation we assume that $j_2=1$. With these assumptions there holds that
\begin{equation}\label{eq:eq20}
c_{K,(L^1_j)} =\Bigg|\frac{\langle f_{1}, h_{L^1_1} \otimes h_{K^2}\rangle}{\langle \sigma_1\rangle_K}
\prod_{j=2}^n\frac{\Big\langle f_{j}, h_{L^1_j} \otimes \frac{1_{K^2}}{|K^2|}\Big\rangle}{\langle \sigma_j\rangle_K}  
\cdot \frac{\Big\langle f_{n+1}, h^0_{K^1}\otimes \frac{1_{K^2}}{|K^2|}\Big\rangle}{\langle \sigma_{n+1} \rangle_K} \Bigg|.
\end{equation}

For $j=2, \dots,n$ we estimate that
\begin{equation}\label{eq:eq21}
\begin{split}
\frac{\Big|\Big\langle f_{j}, h_{L^1_j} \otimes \frac{1_{K^2}}{|K^2|}\Big\rangle\Big|}{\langle \sigma_j\rangle_K}
&= \frac{\Big| \Big\langle \langle f_{j}, h_{L^1_j} \rangle_1 \langle \sigma_j \rangle_{K^1,1}^{-1}\langle \sigma_j \rangle_{K^1,1}, 
\frac{1_{K^2}}{|K^2|}\Big\rangle\Big|}{\langle \langle \sigma_j\rangle_{K^1,1}\rangle_{K^2}}\\
&\le M^{\langle \sigma_j\rangle_{K^1,1}}_{\calD^2}(\langle f_{j}, h_{L^1_j} \rangle_1 \langle \sigma_j \rangle_{K^1,1}^{-1})(x_2)
\end{split}
\end{equation}
for all $x_2 \in K^2$. Also, there holds that
$$
\frac{\Big|\Big\langle f_{n+1}, h^0_{K^1}\otimes \frac{1_{K^2}}{|K^2|}\Big\rangle\Big|}{\langle \sigma_{n+1} \rangle_K}
\le |K^1|^{\frac 12} M_\calD^{\sigma_{n+1}}(f_{n+1}\sigma_{n+1}^{-1})(x)
$$
for all $x \in K$. These give (recall that $L^1_{n+1}=K^1$)  that
\begin{equation*}
\begin{split}
&\sum_{\substack{ L^1_1, \ldots, L_{n+1}^1 \\ (L_j^1)^{(l_j)} = K^1}}
\frac{\prod_{j=1}^{n+1} |L^1_j|^{\frac 12} }{|K^1|^{n+1}} 1_{K^1}\Big( \sum_{K^2} c_{K,(L^1_j)}^{2}
\frac{1_{K^2}}{|K^2|} \Big)^{\frac{1}{2}} \le \prod_{j=1}^n F_{j,K^1} \cdot M_\calD^{\sigma_{n+1}}(f_{n+1}\sigma_{n+1}^{-1}),
\end{split}
\end{equation*}
where
\begin{equation}\label{eq:eq24}
F_{1,K^1}= 1_{K^1}\sum_{(L_1^1)^{(l_1)}=K^1} \frac{|L^1_1|^{\frac 12}}{|K^1|}  \Big( \sum_{K^2}\frac{|\langle f_{1}, h_{L^1_1} \otimes h_{K^2}\rangle|^2}{\langle \sigma_1\rangle_K^2}\frac{1_{K^2}}{|K^2|} \Big)^{\frac{1}{2}}
\end{equation}
and 
\begin{equation}\label{eq:eq25}
F_{j,K^1}
=1_{K^1}\sum_{(L_j^1)^{(l_j)}=K^1} \frac{|L^1_j|^{\frac 12}}{|K^1|}
M^{\langle \sigma_j\rangle_{K^1,1}}_{\calD^2}(\langle f_{j}, h_{L^1_j} \rangle_1 \langle \sigma_j \rangle_{K^1,1}^{-1})
\end{equation}
for $j=2, \dots, n$.

We will now continue from \eqref{eq:eq17} using the above pointwise estimates. Notice that
\begin{equation*}
\begin{split}
\sum_{K^1}\prod_{j=1}^n F_{j,K^1} 
\le \prod_{j=1}^2 \Big(\sum_{K^1} (F_{j,K^1} )^2 \Big)^{1/2}\prod_{j=3}^n \sup_{K^1}F_{j,K^1}
\le \prod_{j=1}^n \Big(\sum_{K^1} (F_{j,K^1} )^2 \Big)^{1/2}.
\end{split}
\end{equation*}
Since $v^{-1}=\prod_{j=1}^{n+1}v_j^{-1}$, we have that
\begin{equation*}
\eqref{eq:eq17} 
\lesssim \prod_{j=1}^n \Big \| \Big( \sum_{K^1} F_{j,K^1}^2 \Big)^{\frac 12} v_{j}^{-1} \Big \|_{L^2} 
\big\|M_\calD^{\sigma_{n+1}}(f_{n+1}\sigma_{n+1}^{-1})v_{n+1}^{-1} \big \|_{L^2}.
\end{equation*}
Since $\sigma_j=v_j^{-2}$ there holds by Proposition \ref{prop:prop1} that
$$
\|M_\calD^{\sigma_{n+1}}(f_{n+1}\sigma_{n+1}^{-1})v_{n+1}^{-1} \|_{L^2}
=\|M_\calD^{\sigma_{n+1}}(f_{n+1}\sigma_{n+1}^{-1}) \|_{L^2(\sigma_{n+1})}
\lesssim \| f_{n+1} v_{n+1} \|_{L^2}.
$$
It remains to estimate the norms for $j=1, \dots, n$.

We begin with $j=1$. If $l_1=0$, then we directly have that
$$
\Big(\sum_{K^1} F_{1,K^1}^2 \Big)^{\frac 12}
=\Big( \sum_{K}\frac{|\langle f_{1}, h_K\rangle|^2}{\langle \sigma_1\rangle_K^2}\frac{1_{K}}{|K|} \Big)^{\frac{1}{2}}.
$$
Since $|\langle f_{1}, h_K\rangle| |K|^{-\frac 12} \le \langle | \Delta_K f_1 | \rangle_K$, we may use Proposition \ref{prop:prop3}
to have that
$$
\Big \| \Big( \sum_{K^1} F_{1,K^1}^2 \Big)^{\frac 12} v_{1}^{-1} \Big \|_{L^2}
\lesssim \| f_1 \sigma_1^{-\frac 12} \|_{L^2}=\| f_1 v_1 \|_{L^2}.
$$

Suppose then $l_1>0$. There holds that
$$
\Big \| \Big( \sum_{K^1} F_{1,K^1}^2 \Big)^{\frac 12} v_{1}^{-1} \Big \|_{L^2}
= \Big( \sum_{K^1} \| F_{1,K^1} v_1^{-1} \|_{L^2}^2\Big)^{\frac 12}.
$$
Let $s \in (1, \infty)$ be such that $d_1/s'=\beta/(2n)$. Then
\begin{equation*}
F_{1,K^1}
\le 2^{\frac{l_1\beta}{2n}}1_{K^1}\bigg(\sum_{(L_1^1)^{(l_1)}=K^1}  \frac{|L^1_1|^{\frac s2}}{|K^1|^s}  \Big( \sum_{K^2}\frac{|\langle f_{1}, h_{L^1_1} \otimes h_{K^2}\rangle|^2}{\langle \sigma_1\rangle_K^2}\frac{1_{K^2}}{|K^2|} \Big)^{\frac{s}{2}} \bigg)^{\frac 1s}.
\end{equation*}
Therefore, $\| F_{1,K^1} v_j^{-1} \|_{L^2}^2$ is less than
\begin{equation}\label{eq:eq18}
2^{\frac{l_1\beta}{n}} \bigg\| \bigg(\sum_{(L_1^1)^{(l_1)}=K^1}  \frac{|L^1_1|^{\frac s2}}{|K^1|^s}  
\Big( \sum_{K^2}\frac{|\langle f_{1}, h_{L^1_1} \otimes h_{K^2}\rangle|^2}{\langle \sigma_1\rangle_K^2}\frac{1_{K^2}}{|K^2|} \Big)^{\frac{s}{2}} 
\bigg)^{\frac 1s} \langle \sigma_1\rangle_{K^1,1}^{\frac 12} \bigg\|_{L^2}^2 |K^1|. 
\end{equation}
Notice that 
$
|\langle f_{1}, h_{L^1_1} \otimes h_{K^2}\rangle | |K^2|^{-\frac 12} 
\le \langle | \Delta_{K^2} \langle f_1, h_{L_1^1} \rangle_1| \rangle_{K^2}.
$
Therefore, the one-parameter case of Proposition \ref{prop:prop3} gives that 
\begin{equation}\label{eq:eq19}
\begin{split}
\eqref{eq:eq18}
&\lesssim 2^{\frac{l_1\beta}{n}} \bigg\| \Big(\sum_{(L_1^1)^{(l_1)}=K^1}  \frac{|L^1_1|^{\frac s2}}{|K^1|^s}  
|\langle f_{1}, h_{L^1_1}\rangle_1|^s\Big)^{\frac 1s} \langle \sigma_1\rangle_{K^1,1}^{-\frac 12} \bigg\|_{L^2}^2 |K^1| \\
& \le   2^{\frac{l_1\beta}{n}} \bigg\| \sum_{(L_1^1)^{(l_1)}=K^1}  \frac{|L^1_1|^{\frac 12}}{|K^1|}  
|\langle f_{1}, h_{L^1_1}\rangle_1| \langle \sigma_1\rangle_{K^1,1}^{-\frac 12} \bigg\|_{L^2}^2 |K^1|.
\end{split}
\end{equation}
Notice that
$$
\sum_{(L_1^1)^{(l_1)}=K^1}  \frac{|L^1_1|^{\frac 12}}{|K^1|}  
|\langle f_{1}, h_{L^1_1}\rangle_1|
\le \langle | \Delta_{K^1,l_1}^1 f_1 | \rangle_{K^1,1}.
$$
Thus, summing the right hand side of \eqref{eq:eq19} over $K^1$ leads to
\begin{equation*}
\begin{split}
 2^{\frac{l_1\beta}{n}} \int_{\R^{d_2}} \sum_{K^1} \langle | \Delta_{K^1,l_1}^1 f_1 | \rangle_{K^1,1}^2 \langle \sigma_1\rangle_{K^1,1}^{-1} |K^1|
 &=2^{\frac{l_1\beta}{n}} \int_{\R^d} \sum_{K^1} \langle | \Delta_{K^1,l_1}^1 f_1 | \rangle_{K^1,1}^2 
 \frac{1_{K^1}}{\langle \sigma_1 \rangle_{K^1,1}^2} \sigma_1 \\
 & \lesssim 2^{\frac{l_1\beta}{n}} \int_{\R^{d}} | f_1 |^2 v_1^2,
 \end{split}
\end{equation*}
where we used Proposition \ref{prop:prop3} again.

Finally, we estimate the norms related to $j=2, \dots, n$, which are all similar. We assume that $l_j>0$. It will be clear how to do the case $l_j=0$.
As above we have
\begin{equation*}
F_{j,K^1}
\le 2^{\frac{l_j\beta}{2n}}1_{K^1}\bigg(\sum_{(L_j^1)^{(l_j)}=K^1} \frac{|L^1_j|^{\frac s2}}{|K^1|^s}
M^{\langle \sigma_j\rangle_{K^1,1}}_{\calD^2}(\langle f_{j}, h_{L^1_j} \rangle_1 \langle \sigma_j \rangle_{K^1,1}^{-1})^s \bigg)^{\frac 1s}.
\end{equation*}
Therefore, we get that
\begin{equation*}
\begin{split}
\| F_{j,K^1} v_{j}^{-1}\|_{L^2}^2
&\le 2^{\frac{l_j\beta}{n}} \bigg\| \bigg(\sum_{(L_j^1)^{(l_j)}=K^1} \frac{|L^1_j|^{\frac s2}}{|K^1|^s}
M^{\langle \sigma_j\rangle_{K^1,1}}_{\calD^2}(\langle f_{j}, h_{L^1_j} \rangle_1 \langle \sigma_j \rangle_{K^1,1}^{-1})^s 
\bigg)^{\frac 1s} \langle \sigma_j \rangle_{K^1,1}^{\frac 12} \bigg \|_{L^2}^2|K^1| \\
&\lesssim 2^{\frac{l_j\beta}{n}}
\bigg\| \bigg(\sum_{(L_j^1)^{(l_j)}=K^1} \frac{|L^1_j|^{\frac s2}}{|K^1|^s}
|\langle f_{j}, h_{L^1_j} \rangle_1 \langle \sigma_j \rangle_{K^1,1}^{-1}|^s 
\bigg)^{\frac 1s} \langle \sigma_j \rangle_{K^1,1}^{\frac 12} \bigg \|_{L^2}^2|K^1| \\
& \le 2^{\frac{l_j\beta}{n}}
\bigg\| \sum_{(L_j^1)^{(l_j)}=K^1} \frac{|L^1_j|^{\frac 12}}{|K^1|}
|\langle f_{j}, h_{L^1_j} \rangle_1| 
 \langle \sigma_j \rangle_{K^1,1}^{-\frac 12} \bigg \|_{L^2}^2|K^1|,
\end{split}
\end{equation*}
where we applied the one-parameter version of Proposition \ref{prop:vecvalmax}. The last norm is
like the last norm in \eqref{eq:eq19}, and therefore the estimate can be concluded with familiar steps.
Combining the estimates we have shown that
$$
\prod_{j=1}^n \Big \| \Big( \sum_{K^1} F_{j,K^1}^2 \Big)^{\frac 12} v_{j}^{-1} \Big \|_{L^2}
\lesssim \prod_{j=1}^n 2^{\frac{l_j\beta}{2n}} \| f_j v_j \|_{L^2}
\le 2^{\max k_j\frac{\beta}{2}} \prod_{j=1}^n  \| f_j v_j \|_{L^2}.
$$

Above, we assumed that the index $j_2$ related to the form of the paraproduct satisfied $j_2=1$, see the discussion before \eqref{eq:eq20}.
It remains to comment on the case $j_2=n+1$. In this case the formula corresponding to \eqref{eq:eq20} is
$$
c_{K,(L^1_j)} =\Bigg|\prod_{j=1}^n\frac{\Big\langle f_{j}, h_{L^1_j} \otimes \frac{1_{K^2}}{|K^2|}\Big\rangle}{\langle \sigma_j\rangle_K}  
\cdot \frac{\langle f_{n+1}, h^0_{K^1}\otimes h_{K^2}\rangle}{\langle \sigma_{n+1} \rangle_K} \Bigg|.
$$
For $j=1, \dots, n$ we do the estimate \eqref{eq:eq21}. Also, there holds that
\begin{equation*}
\begin{split}
\frac{|\langle f_{n+1}, h^0_{K^1}\otimes h_{K^2}\rangle |}{\langle \sigma_{n+1} \rangle_K}
&= |K^1|^{\frac 12}\frac{\big|\bla \langle f_{n+1}, h_{K^2}\rangle_2 \langle \sigma_{n+1} \rangle_{K^2,2}^{-1}\langle \sigma_{n+1} \rangle_{K^2,2} \bra_{K^1}\big|}
{\langle \langle \sigma_{n+1} \rangle_{K^2,2} \rangle_{K^1}} \\
& \le |K^1|^{\frac 12} M_{\calD^1}^{\langle \sigma_{n+1} \rangle_{K^2,2}}(\langle f_{n+1}, h_{K^2}\rangle_2 \langle \sigma_{n+1} \rangle_{K^2,2}^{-1})(x_1)
\end{split}
\end{equation*}
for any $x_1 \in K^1$. 
With the pointwise estimates we proceed as above. Related to $f_j$, $j=1, \dots, n$, this leads to terms which we know how to estimate. 
Related to $f_{n+1}$ we get a similar term except that the parameters are in opposite roles.
We are done.
\end{proof}

\subsection{Full paraproducts}
An $n$-linear bi-parameter full paraproduct $\Pi$ takes the form
\begin{equation}\label{eq:pi2bar}
\langle \Pi(f_1, \ldots, f_n) , f_{n+1} \rangle = \sum_{K = K^1 \times K^2} a_{K} \prod_{j=1}^{n+1} \langle f_j, u_{j, K^1} \otimes u_{j, K^2} \rangle,
\end{equation}
where the functions $u_{j, K^1}$ and $u_{j, K^2}$ are like in \eqref{eq:Spi}.
The coefficients are assumed to satisfy
$$
\| (a_{K} ) \|_{\BMO_{\operatorname{prod}}} = \sup_{\Omega} \Big(\frac{1}{|\Omega|} \sum_{K\subset \Omega} |a_{K}|^2 \Big)^{1/2} \le 1,
$$
where the supremum is over open sets $\Omega \subset \R^d = \R^{d_1} \times \R^{d_2}$ with $0 < |\Omega| < \infty$.
As already discussed the $H^1$-$\BMO$ duality works also in bi-parameter (see again \cite[Equation (4.1)]{MO}):
\begin{equation}\label{eq:BiParH1BMO}
\sum_{K} |a_{K}| |b_{K}| \lesssim \| (a_{K} ) \|_{\BMO_{\operatorname{prod}}} \Big\| \Big( \sum_{K} |b_{K}|^2 \frac{1_{K}}{|K|} \Big)^{1/2} \Big \|_{L^1}.
\end{equation}

We are ready to bound the full paraproducts.
\begin{thm}
Suppose $\Pi$ is an $n$-linear bi-parameter full paraproduct, $1 < p_1, \ldots, p_n \le \infty$ and $1/p = \sum_{i=1}^n 1/p_i> 0$. Then we have
$$
\|\Pi(f_1, \ldots, f_n)w \|_{L^p} \lesssim \prod_{i=1}^n \|f_i w_i\|_{L^{p_i}} 
$$
for all multilinear bi-parameter weights $\vec w \in A_{\vec p}$. 
\end{thm}
\begin{proof}
We use duality to always reduce to one of the operators of type $A_{1}$ in Theorem \ref{thm:thm3}.
Fix some $\vec p = (p_1, \ldots, p_n)$ with $1 < p_i < \infty$ and $p > 1$, which is enough by extrapolation. We will dualise using $f_{n+1}$ with $\|f_{n+1}w^{-1}\|_{L^{p'}} \le 1$. The particular form of $\Pi$ does not matter -- it only affects the form of the operator $A_1$ we will get.
We may, for example, look at
$$
\Pi(f_1, \ldots, f_n) = \sum_{K = K^1 \times K^2} a_K \Big \langle f_{1}, h_{K^1} \otimes \frac{1_{K^2}}{|K^2|} \Big\rangle
 \Big \langle f_{2}, \frac{1_{K^1}}{|K^1|} \otimes h_{K^2} \Big\rangle  \prod_{j=3}^{n} \langle f_j \rangle_K \cdot \frac{1_K}{|K|}.
$$
We have
\begin{align*}
|\langle \Pi(f_1, \ldots, f_n), f_{n+1} \rangle| \le \sum_{K} |a_K| \Big| \Big \langle f_{1}, h_{K^1} \otimes \frac{1_{K^2}}{|K^2|} \Big\rangle
 \Big \langle f_{2}, \frac{1_{K^1}}{|K^1|} \otimes h_{K^2} \Big\rangle \Big|  \prod_{j=3}^{n+1} \langle |f_j| \rangle_K.
\end{align*}
We now apply the unweighted $H^1$-$\BMO$ duality estimate from above to bound this with
\begin{align*}
\Big\| \Big( \sum_{K} \langle | \Delta_{K^1}^1 f_1 | \rangle_K^2
 \langle | \Delta_{K^2}^2 f_2 | \rangle_K^2  \prod_{j=3}^{n+1} \langle |f_j| \rangle_K^2 1_K \Big)^{\frac{1}{2}} \Big\|_{L^1}.
\end{align*}
Recalling \eqref{eq:eq12} it remains to apply Theorem \ref{thm:thm3} with a suitable $A_1(f_1, \ldots, f_{n+1})$.
\end{proof}

\section{Singular integrals}\label{sec:SIOs}
Let $\omega$ be a modulus of continuity: an increasing and subadditive function with $\omega(0) = 0$. 
A relevant quantity is
the modified Dini condition
\begin{equation}\label{eq:Dini}
\|\omega\|_{\operatorname{Dini}_{\alpha}} := \int_0^1 \omega(t) \Big( 1 + \log \frac{1}{t} \Big)^{\alpha} \frac{dt}{t}, \qquad \alpha \ge 0.
\end{equation}
In practice, the quantity \eqref{eq:Dini} arises as follows:
\begin{equation}\label{eq:diniuse}
\sum_{k=1}^{\infty} \omega(2^{-k}) k^{\alpha} = \sum_{k=1}^{\infty} \frac{1}{\log 2} \int_{2^{-k}}^{2^{-k+1}} \omega(2^{-k}) k^{\alpha} \frac{dt}{t} \lesssim \int_0^1 \omega(t) \Big( 1 + \log \frac{1}{t} \Big)^{\alpha} \frac{dt}{t}.
\end{equation}

We define what it means to be an $n$-linear bi-parameter SIO. 
Let $\scrF_{d_i}$ denote the space of finite linear combinations of indicators of cubes in $\R^{d_i}$, and let
$\scrF$ denote the space of finite linear combinations of indicators of rectangles in $\R^{d}$. 
Suppose that we have $n$-linear operators $T^{j_1*,j_2*}_{1,2}$, $j_1,j_2 \in \{0, \dots, n\}$, each mapping
$\scrF \times \dots \times \scrF$ into locally integrable functions. We denote $T=T^{0*,0*}$ and assume that the 
operators $T^{j_1*,j_2*}_{1,2}$ satisfy the duality relations as described in Section \ref{sec:adjoints}.

Let $\omega_i$ be a modulus of continuity on $\R^{d_i}$.
Assume $f_j = f_j^1 \otimes f_j^2$, $j = 1, \ldots, n+1$, where $f_{j}^i \in \scrF_{d_i}$.

\subsection*{Bi-parameter SIOs}
\subsubsection*{Full kernel representation}
Here we assume that given $m \in \{1,2\}$ there exist $j_1, j_2 \in \{1, \ldots, n+1\}$ so that
$\operatorname{spt} f_{j_1}^m \cap \operatorname{spt} f_{j_2}^m = \emptyset$.
In this case we demand that
$$
\langle T(f_1, \ldots, f_n), f_{n+1}\rangle = \int_{\R^{(n+1)d}}  K(x_{n+1},x_1, \dots, x_n)\prod_{j=1}^{n+1} f_j(x_j) \ud x,
$$
where
$$
K \colon \R^{(n+1)d} \setminus \{ (x_1, \ldots, x_{n+1}) \in \R^{(n+1)d}\colon x_1^1 = \cdots =  x_{n+1}^1 \textup{ or }  x_1^2 = \cdots =  x_{n+1}^2\} \to \C
$$
is a kernel satisfying a set of estimates which we specify next. 

The kernel $K$ is assumed to satisfy the size estimate
\begin{displaymath}
|K(x_{n+1},x_1, \dots, x_n)| \lesssim \prod_{m=1}^2 \frac{1}{\Big(\sum_{j=1}^{n} |x_{n+1}^m-x_j^m|\Big)^{d_mn}}.
\end{displaymath}

We also require the following continuity estimates. For example, we require that we have
\begin{align*}
|K(x_{n+1}, x_1, \ldots, x_n)-&K(x_{n+1},x_1, \dots, x_{n-1}, (c^1,x^2_n))\\
&-K((x_{n+1}^1,c^2),x_1, \dots, x_n)+K((x_{n+1}^1,c^2),x_1, \dots, x_{n-1},  (c^1,x^2_n))| \\
&\qquad \lesssim \omega_1 \Big( \frac{|x_{n}^1-c^1| }{ \sum_{j=1}^{n} |x_{n+1}^1-x_j^1|} \Big) 
\frac{1}{\Big(\sum_{j=1}^{n} |x_{n+1}^1-x_j^1|\Big)^{d_1n}} \\
&\qquad\times
\omega_2 \Big( \frac{|x_{n+1}^2-c^2| }{ \sum_{j=1}^{n} |x_{n+1}^2-x_j^2|} \Big) 
\frac{1}{\Big(\sum_{j=1}^{n} |x_{n+1}^2-x_j^2|\Big)^{d_2n}}
\end{align*}
whenever $|x_n^1-c^1| \le 2^{-1} \max_{1 \le i \le n} |x_{n+1}^1-x_i^1|$
and $|x_{n+1}^2-c^2| \le 2^{-1} \max_{1 \le i \le n} |x_{n+1}^2-x_i^2|$.
Of course, we also require all the other natural symmetric estimates, where $c^1$ can be in any of the given $n+1$ slots and similarly for $c^2$. There
are $(n+1)^2$ different estimates.

Finally, we require the following mixed continuity and size estimates. For example, we ask that
\begin{align*}
|K(x_{n+1}&, x_1, \ldots, x_n)-K(x_{n+1},x_1, \dots, x_{n-1}, (c^1,x^2_n))| \\
& \lesssim \omega_1 \Big( \frac{|x_{n}^1-c^1| }{ \sum_{j=1}^{n} |x_{n+1}^1-x_j^1|} \Big) 
\frac{1}{\Big(\sum_{j=1}^{n} |x_{n+1}^1-x_j^1|\Big)^{d_1n}} \cdot  \frac{1}{\Big(\sum_{j=1}^{n} |x_{n+1}^2-x_j^2|\Big)^{d_2n}}
\end{align*}
whenever $|x_n^1-c^1| \le 2^{-1} \max_{1 \le i \le n} |x_{n+1}^1-x_i^1|$. Again, we also require all the other natural symmetric estimates.
\subsubsection*{Partial kernel representations}
Suppose now only that there exist $j_1, j_2 \in \{1, \ldots, n+1\}$ so that
$\operatorname{spt} f_{j_1}^1 \cap \operatorname{spt} f_{j_2}^1 = \emptyset$.
 Then we assume that
$$
\langle T(f_1, \ldots, f_n), f_{n+1}\rangle = \int_{\R^{(n+1)d_1}} K_{(f_j^2)}(x_{n+1}^1, x_1^1, \ldots, x_n^1) \prod_{j=1}^{n+1} f_j^1(x^1_j) \ud x^1,
$$
where $K_{(f_j^2)}$ is a one-parameter $\omega_1$-Calder\'on--Zygmund kernel but with a constant depending on the fixed functions $f_1^2, \ldots, f_{n+1}^2$.
For example, this means that the size estimate takes the form
$$
|K_{(f_j^2)}(x_{n+1}^1, x_1^1, \ldots, x_n^1)| \le C(f_1^2, \ldots, f_{n+1}^2) \frac{1}{\Big(\sum_{j=1}^{n} |x_{n+1}^1-x_j^1|\Big)^{d_1n}}.
$$
The continuity estimates are analogous.

We assume the following $T1$ type control on the constant $C(f_1^2, \ldots, f_{n+1}^2)$. We have
\begin{equation*}\label{eq:PKWBP}
C(1_{I^2}, \ldots, 1_{I^2}) \lesssim |I^2|
\end{equation*}
and
\begin{equation}\label{eq:pest}
C(a_{I^2}, 1_{I^2}, \ldots, 1_{I^2}) + C(1_{I^2}, a_{I^2}, 1_{I^2}, \ldots, 1_{I^2}) + \cdots + C(1_{I^2}, \ldots, 1_{I^2}, a_{I^2}) \lesssim |I^2|
\end{equation}
for all cubes $I^2 \subset \R^{d_2}$
and all functions $a_{I^2} \in \scrF_{d_2}$ satisfying $a_{I^2} = 1_{I^2}a_{I^2}$, $|a_{I^2}| \le 1$ and $\int a_{I^2} = 0$.

Analogous partial kernel representation on the second parameter is assumed when $\operatorname{spt} f_{j_1}^2 \cap \operatorname{spt} f_{j_2}^2 = \emptyset$
for some $j_1, j_2$.

\begin{defn}
If $T$ is an $n$-linear operator with full and partial kernel representations as defined above, we call $T$ an $n$-linear bi-parameter $(\omega_1, \omega_2)$-SIO.
\end{defn}

\subsection*{Bi-parameter CZOs}
We say that $T$ satisfies the weak boundedness property if
\begin{equation*}\label{eq:2ParWBP}
|\langle T(1_R, \ldots, 1_R), 1_R \rangle| \lesssim |R|
\end{equation*}
for all rectangles $R = I^1 \times I^2 \subset \R^{d} = \R^{d_1} \times \R^{d_2}$.

An SIO $T$ satisfies the diagonal BMO assumption if the following holds. For all rectangles $R = I^1 \times I^2 \subset \R^{d} = \R^{d_1} \times \R^{d_2}$
and functions $a_{I^i}\in \scrF_{d_i}$ with $a_{I^i} = 1_{I^i}a_{I^i}$, $|a_{I^i}| \le 1$ and $\int a_{I^i} = 0$ we have
\begin{equation*}\label{eq:DiagBMO}
|\langle T(a_{I^1} \otimes 1_{I^2}, 1_R, \ldots, 1_R), 1_R \rangle| + \cdots +  |\langle T(1_R, \ldots, 1_R), a_{I^1} \otimes 1_{I^2} \rangle| \lesssim |R|
\end{equation*}
and
$$
|\langle T(1_{I^1} \otimes a_{I^2}, 1_R, \ldots, 1_R), 1_R \rangle| + \cdots +  |\langle T(1_R, \ldots, 1_R), 1_{I^1} \otimes a_{I^2} \rangle| \lesssim |R|.
$$

An SIO $T$ satisfies the product BMO assumption if it holds
$$S(1, \ldots, 1) \in \BMO_{\textup{prod}}$$ for all the $(n+1)^2$ adjoints $S = T^{j_1*, j_2*}_{1,2}$.
This can be interpreted in the sense that
$$
\| S(1, \ldots, 1) \|_{\BMO_{\operatorname{prod}}} = \sup_{\calD = \calD^1 \times \calD^2} \sup_{\Omega} \Big(\frac{1}{|\Omega|} \sum_{ \substack{ R = I^1 \times I^2 \in \calD \\
R \subset \Omega}} |\langle S(1, \ldots, 1), h_R \rangle|^2 \Big)^{1/2} < \infty,
$$
where the supremum is over all dyadic grids $\calD^i$ on $\R^{d_i}$ and
open sets $\Omega \subset \R^d = \R^{d_1} \times \R^{d_2}$ with $0 < |\Omega| < \infty$, and the pairings
$\langle S(1, \ldots, 1), h_R\rangle$ can be defined, in a natural way, using the kernel representations.

\begin{defn}\label{defn:CZO}
An $n$-linear  bi-parameter $(\omega_1, \omega_2)$-SIO $T$
satisfying the weak boundedness property, the diagonal BMO assumption and the product BMO assumption is called an $n$-linear bi-parameter
$(\omega_1, \omega_2)$-Calder\'on--Zygmund operator ($(\omega_1, \omega_2)$-CZO). 
\end{defn}

\subsection*{Dyadic representation theorem}
In Section \ref{sec:dmo} we have introduced the three different dyadic model operators (DMOs).
In this section we explain how and why the DMOs are linked to the CZOs. Before
stating the known representation theorem, to aid the reader, we first outline the main structure and idea of representation theorems -- for the lenghty details
in this generality see \cite{AMV}.

\textbf{Step 1.}
There is a natural probability space $\Omega = \Omega_1 \times \Omega_2$, the details of which are not relevant for us here (but see \cite{Hy1}),
so that to each $\sigma = (\sigma_1, \sigma_2) \in \Omega$
we can associate a random collection of dyadic rectangles $\calD_{\sigma} = \calD_{\sigma_1} \times \calD_{\sigma_2}$. 
The starting point is the martingale difference decomposition
\begin{equation*}
\langle T(f_1, \ldots, f_n),f_{n+1} \rangle
= \sum_{j_1, j_2 =1}^{n+1}  \E_{\sigma} \Sigma_{j_1, j_2, \sigma} + \E_{\sigma} \operatorname{Rem}_{\sigma},
\end{equation*}
where
$$
\Sigma_{j_1, j_2, \sigma} = \sum_{ \substack{ R_1, \ldots, R_{n+1} \\ \ell(I_{i_1}^1) > \ell(I_{j_1}^1) \textup{ for } i_1 \ne j_1
 \\ \ell(I_{i_2}^2) > \ell(I_{j_2}^2) \textup{ for } i_2 \ne j_2}} \langle T(\Delta_{R_1}f_1, \ldots, \Delta_{R_n}f_n),\Delta_{R_{n+1}}f_{n+1} \rangle
$$
and $R_1 = I_1^1 \times I_1^2, \ldots, R_{n+1} = I_{n+1}^1 \times I_{n+1}^2 \in \calD_\sigma = \calD_{\sigma_1} \times \calD_{\sigma_2}$.
Notice how we have already started the proof working parameter by parameter.
At this point the randomization is not yet important: it is used at a later point in the proof to find suitable common parents
for dyadic cubes. Looking at the definition of shifts this is clearly critical: everything is organised under
the common parent $K$ and they cannot be arbitrarily large.

\textbf{Step 2.} There are $(n+1)^2$ main terms $\Sigma_{j_1, j_2, \sigma}$ -- these are similar to each
other and all of them produce shifts, partial paraproducts and \emph{exactly one} full paraproduct.
For example, a further parameter by parameter $T1$ style decomposition of $\Sigma_{n, n+1, \sigma}$ produces the full paraproduct
\begin{align*}
&\sum_{R = K^1 \times K^2} \langle T(1, \ldots, 1, h_{K^1} \otimes 1), &1 \otimes h_{K^2} \rangle
\prod_{j=1}^{n-1} \langle f_j \rangle_{R}\Big \langle f_n, h_{K^1} \otimes \frac{1_{K^2}}{|K^2|} \Big\rangle
\Big \langle f_{n+1}, \frac{1_{K^1}}{|K^1|} \otimes h_{K^2} \Big\rangle,
\end{align*}
where
$$
\langle T(1, \ldots, 1, h_{K^1} \otimes 1), 1 \otimes h_{K^2} \rangle =  \langle T^{n*}_1(1, \ldots, 1), h_R \rangle.
$$
For the quantitative part the product $\BMO$ assumption of $T^{n*}_1$ is critical here, and the remaining
product $\BMO$ assumptions are used to control the full paraproducts coming
from the other main terms.

The shifts and partial paraproducts structurally arise from the $T1$ decomposition combined with
probability. Again, the randomization is simply used to find suitably sized common parents. After this completely
structural part (for full details see \cite{AMV}), the focus is on providing
estimates for the coefficients, like the coefficient $a_{K, (R_j)}$ of the shifts.
\emph{As in the full paraproduct case above, it is important to understand that the coefficients
always have a concrete form in terms of pairings involving $T$ and various Haar functions.}
These pairings are estimated in various ways:
\begin{itemize}
	\item The shift coefficients are handled with kernel estimates only (various size and continuity estimates).
	Only in the part $\operatorname{Rem}_{\sigma}$, which we have not yet discussed,
	also the weak boundedness is used to handle the diagonal case, where kernel
	estimates are not valid.
	\item In the partial paraproduct case a size estimate does not suffice, as there is not enough
	cancellation. A more refined $\BMO$ estimate needs to be proved -- this is done via a duality argument.
	This duality is the source of the atoms $a_I$ appearing in some of the assumptions -- e.g. in \eqref{eq:pest}.
\end{itemize}

\textbf{Step 3.}
The final step is to deal with the remainder $\operatorname{Rem}_{\sigma}$. This only
produces shifts and partial paraproducts.
Another difference to the main terms is that all the diagonal parts of the summation are here -- to deal
with them we need to assume the weak boundedness property and the diagonal $\BMO$ assumptions.

\begin{rem}\label{rem:mrem}
	An $m$-parameter representation theorem is structurally identical: the pairing
	$\langle T(f_1, \ldots, f_n),f_{n+1} \rangle$ is split into $(n+1)^m$ main terms
	and the remainder. These are then further split into shifts, partial paraproducts and full paraproducts.
	The full paraproduct is produced parameter by parameter, as it is in the bi-parameter case, and this produces partial paraproducts, where the
	paraproduct component can vary from being $1$-parameter to being $(m-1)$-parameter.
	The definition of a CZO is adjusted so that all of the appearing
	coefficients of the appearing model operators involving $T$ and Haar functions can be estimated.
	For the linear $m$-parameter representation theorem see Ou \cite{Ou} -- this establishes the appropriate
	definition of a multi-parameter CZO.
	We discuss the $m$-parameter case in more detail in Section \ref{sec:multi}. The point there is the following:
	while the representation theorem itself is straightforward, some
	of the estimates of Section \ref{sec:dmo} are harder in $m$-parameter. 
\end{rem}

In the paper \cite{AMV}, among other things, a dyadic representation theorem for $n$-linear bi-parameter CZOs was proved.
The minimal regularity required is $\omega_i \in \operatorname{Dini}_{\frac{1}{2}}$, but then the dyadic representation is in terms
of certain modified versions of the model operators we have presented, and bounded, in this paper. It appears to be difficult to prove weighted bounds for the modified operators
with the optimal dependency on the complexity. Instead, we will rely on a lemma, which says that all of the modified operators can be written
as suitable sums of the standard model operators. This step essentially outright loses $\frac{1}{2}$ of kernel regularity, and puts
us in competition to obtain our weighted bounds with $\omega_i \in \operatorname{Dini}_{1}$.
The bilinear bi-parameter representation theorem with the usual H\"older type kernel regularity $w_i(t) = t^{\alpha_i}$ appeared first in \cite{LMV}. We now state a representation theorem that we will rely on.

A consequence of \cite[Theorem 5.35]{AMV} and \cite[Lemma 5.12]{AMV} is the following. 
\begin{prop}
Suppose $T$ is an $n$-linear bi-parameter $(\omega_1, \omega_2)$-CZO.
Then we have
$$
\langle T(f_1,\ldots,f_n), f_{n+1} \rangle= C_T \E_{\sigma} \sum_{u = (u_1, u_2) \in \N^2}  \omega_1(2^{-u_1})\omega_2(2^{-u_2}) \langle U_{u, \sigma}(f_1,\ldots,f_n), f_{n+1} \rangle,
$$
where $C_T$ enjoys a linear bound with respect to the CZO quantities and
$U_{u, \sigma}$ denotes some $n$-linear bi-parameter dyadic operator (defined in the grid $\calD_{\sigma}$) with the following property. We have that $U_u = U_{u, \sigma}$ can be 
decomposed using the standard dyadic model operators as follows:
\begin{equation}\label{eq:eq10}
U_{u}
= C \sum_{i_1=0}^{u_1-1} \sum_{i_2=0}^{u_2-1} V_{i_1,i_2},
\end{equation}
where each $V = V_{i_1,i_2}$ is a dyadic model operator (a shift, a partial paraproduct or a full paraproduct)
of complexity $k^m_{j, V}$, $j \in \{1, \ldots, n+1\}$, $m \in \{1,2\}$,
satisfying 
$$
k^{m}_{j, V} \le u_m.
$$
\end{prop}

\begin{rem}
We assumed that the operator $T$ and its adjoints are initially well-defined for finite linear combinations of indicators of rectangles. 
However, a careful proof of the representation theorem \cite{LMV} shows that this implies the boundedness of $T$ (for related details see also \cite{GH} and \cite{Hy3}).
Therefore, we do not need to worry about this detail any more at this point and we can work with general functions.
Moreover, we do not need to work with the CZOs directly --
after the representation theorem we only need to work with the dyadic model operators.
\end{rem}

\subsection*{Weighted estimates for CZOs}
In this paper we were able to prove a complexity free weighted estimate for the shifts. On the contrary, the weighted estimate for the partial paraproducts is
even exponential, however, with an arbitrarily small power. For these reasons, we can prove a weighted estimate with mild kernel regularity
for paraproduct free $T$, and otherwise we will deal with the standard kernel regularity $\omega_i(t) = t^{\alpha_i}$.
By paraproduct free we mean that the paraproducts in the dyadic representation of $T$ vanish, which could also be stated in terms of (both partial and full) ``$T1=0$'' type conditions (only the partial paraproducts, and not the full paraproducts, are problematic in terms of kernel regularity, of course).  In the paraproduct free case the reader can think of convolution form SIOs.
\begin{thm}
Suppose $T$ is an $n$-linear bi-parameter $(\omega_1, \omega_2)$-CZO.
For $1 < p_1, \ldots, p_n \le \infty$ and $1/p = \sum_i 1/p_i> 0$ we have
$$
\|T(f_1, \ldots, f_n)w \|_{L^p} \lesssim \prod_i \|f_i w_i\|_{L^{p_i}} 
$$
for all multilinear bi-parameter weights $\vec w \in A_{\vec p}$, if one of the following conditions hold.
\begin{enumerate}
\item $T$ is paraproduct free and $\omega_i \in \operatorname{Dini}_{1}$.
\item We have $\omega_i(t) = t^{\alpha_i}$ for some $\alpha_i \in (0,1]$.
\end{enumerate}
\end{thm}
\begin{proof}
Notice that in the paraproduct free case (1) by our results for the shifts we always have
$$
\|U_{u, \sigma}(f_1, \ldots, f_n)w \|_{L^p} \lesssim (1+u_1)(1+u_2)  \prod_i \|f_i w_i\|_{L^{p_i}},
$$
where the complexity dependency comes only from the decomposition \eqref{eq:eq10}.
We then take some $1 < p_1, \ldots, p_n < \infty$ with $p \in (1, \infty)$, use the dyadic representation theorem and conclude that
$T$ satisfies the weighted bound with these fixed exponents -- recall \eqref{eq:diniuse} and that $\omega_i \in \operatorname{Dini}_{1}$.
Finally, we extrapolate using Theorem \ref{thm:ext}.

The case of a completely general CZO with the standard kernel regularity is proved completely analogously. Just choose the exponent $\beta$
in the exponential complexity dependendency of the partial paraproducts to be small enough compared to $\alpha_1$ and $\alpha_2$.
\end{proof}

\section{Extrapolation}\label{app:app1}
This section is devoted to providing more details about Theorem \ref{thm:ext} in the multi-parameter setting. We also obtain the proof of the vector-valued
Proposition \ref{prop:vecvalmax}. We give the details in the bi-parameter case with the general case being similar.

 We begin with the following definitions. Given $\mu\in A_\infty(\R^{n+m})$, we say $w\in A_p(\mu)$ if $w>0$ a.e. and
\[
[w]_{A_p(\mu)}:=\sup_R\,  \langle w\rangle_R^{\mu} \left(\big\langle w^{-\frac 1{p-1}}\big\rangle_R^{\mu}\right)^{p-1}<\infty,\qquad 1<p<\infty.
\]
And we say $w\in A_1(\mu)$ if $w>0$ a.e. and
\[
[w]_{A_1(\mu)} = \sup_R \, \ave{w}_R^\mu \esssup_R w^{-1}<\infty.
\]

We begin with the following auxiliary result needed to build the required machinery.
This is an extension of Proposition \ref{prop:prop1}.
\begin{rem}
The so-called
three lattice theorem states that
there are lattices $\calD^m_j$ in $\R^{d_m}$, $m \in \{1,2\}$, $j \in \{1, \ldots, 3^{d_m}\}$, such that for every cube $Q^m \subset \R^{d_m}$
there exists a $j$ and $I^m \in \calD^m_j$ so that $Q^m \subset I^m$ and $|I^m| \sim |Q^m|$.
Given $\lambda \in A_\infty$ we have in particular that $\lambda$ is doubling: $\lambda(2R) \lesssim \lambda(R)$ for all rectangles $R$.
It then follows that also the non-dyadic variant $M^{\lambda}$ satisfies Proposition \ref{prop:prop1}.
\end{rem}
\begin{lem}\label{lem:lem5}
Let $\mu\in A_\infty$ and $w\in A_p(\mu)$, $1<p<\infty$. Then we have
\[
\| M^\mu f\|_{L^p(w\mu)}\lesssim \|f\|_{L^p(w\mu)}.
\]
\end{lem}
\begin{proof}
Fix $x$ and $f\ge 0$ and denote $\sigma=w^{-\frac 1{p-1}}$. For an arbitrary rectangle $R \subset \R^d$ with $x\in R$ we have
\begin{align*}
\langle f \rangle_R^\mu &= \langle \sigma\rangle_R^\mu \left( \langle w\rangle_R^\mu\right)^{\frac 1{p-1}}\left( \langle w\rangle_R^\mu\right)^{-\frac 1{p-1}}\frac 1{\sigma\mu(R)}\int_R f\mu\\
&\le [w]_{A_p(\mu)}^{\frac 1{p-1}}    \left(M^{w\mu}    \big(   [M^{\sigma \mu}(f \sigma^{-1})]^{p-1} w^{-1} \big)(x)\right)^{\frac 1{p-1}}.
\end{align*}
The idea of the above pointwise estimate is from \cite{Le}.
If $w\mu, \sigma \mu\in A_\infty$, then by (the non-dyadic version of) Proposition \ref{prop:prop1} we have 
\begin{align*}
\| M^{\mu} f\|_{L^p(w\mu)}&\lesssim \left\| \left(M^{w\mu} \big(   [M^{\sigma \mu}(f \sigma^{-1})]^{p-1} w^{-1} \big) \right)^{\frac 1{p-1}}\right\|_{L^p(w\mu)}\\
&\lesssim \left\| \left(    [M^{\sigma \mu}(f \sigma^{-1})]^{p-1} w^{-1}   \right)^{\frac 1{p-1}}\right\|_{L^p(w\mu)}\\
&= \| M^{\sigma \mu}(f \sigma^{-1})\|_{L^p(\sigma\mu)}\lesssim \|f\|_{L^p(w\mu)}.
\end{align*}

Therefore, it remains to check that $w\mu, \sigma \mu\in A_\infty$. We only explicitly prove that $w\mu\in A_\infty$, since the other one is symmetric. First of all, write 
\[
\langle w\rangle_R^{\mu} \left(\big\langle w^{-\frac 1{p-1}}\big\rangle_R^{\mu}\right)^{p-1}\le [w]_{A_p(\mu)}
\]
in the form
\[
\langle w\mu\rangle_R \big\langle w^{-\frac 1{p-1}}\mu\big\rangle_R^{p-1}\le  [w]_{A_p(\mu)} \langle \mu\rangle_R^p.
\]
Then by the Lebesgue differentiation theorem, we have for all cubes $I^1 \subset \R^{d_1}$ that
\[
\langle w\mu\rangle_{I^1,1}(x_2) \big\langle w^{-\frac 1{p-1}}\mu\big\rangle_{I^1,1}^{p-1}(x_2) \le  [w]_{A_p(\mu)} \langle \mu\rangle_{I^1,1}^p(x_2),\quad x_2\in \R^{d_2}\setminus N_{I^1},
\]where $|N_{I^1}|=0$. By standard considerations there exists $N$ so that $|N| = 0$ and for all cubes $I^1 \subset \R^{d_1}$ we have
\[
\langle w\mu\rangle_{I^1,1}(x_2) \big\langle w^{-\frac 1{p-1}}\mu\big\rangle_{I^1,1}^{p-1}(x_2) \le  [w]_{A_p(\mu)} \langle \mu\rangle_{I^1,1}^p(x_2) ,\quad x_2\in \R^{d_2}\setminus N.
\]
In other words, $w(\cdot, x_2)\in A_p(\mu(\cdot, x_2))$ (uniformly) for all $x_2\in \R^m\setminus N$.
As $\mu \in A_{\infty}$ there exists $s < \infty$ so that $\mu \in A_s$.
 Then for all cubes
$I^1 \subset \R^{d_1}$ and arbitrary $E\subset I^1$ we have
\begin{align*}
\frac{|E|}{|I^1|}\lesssim \Big(\frac{\mu(\cdot, x_2)(E)}{\mu(\cdot, x_2)(I^1)} \Big)^{\frac{1}{s}}
\lesssim \Big(\frac{w\mu(\cdot, x_2)(E)}{w\mu(\cdot, x_2)(I^1)}\Big)^{\frac{1}{ps}},\quad {\rm {a.e.}} \, x_2\in \R^{d_2},
\end{align*}where the implicit constant is independent from $x_2$.
This means $w\mu(\cdot, x_2)\in A_\infty(\R^{d_1})$ uniformly for a.e. $x_2\in \R^{d_2}$. Likewise we can show that $w\mu(x_1,\cdot)\in A_\infty(\R^{d_2})$ uniformly for a.e. $x_1\in \R^{d_1}$. This completes the proof and we are done.
\end{proof}

Now we are ready to formulate the following version of Rubio de Francia algorithm.  
\begin{lem}
Let $\mu \in A_{\infty}$ and $p \in (1,\infty)$. Let $f$ be a non-negative function in $L^p(w\mu)$ for some $w\in A_p(\mu)$. Let $M^\mu_k$ be the $k$-th iterate of $M^\mu$, $M^\mu_0f=f$, and $\|M^\mu\|_{L^p(w\mu)} := \|M^\mu\|_{L^p(w\mu) \to L^p(w\mu)}$ be the norm of $M^\mu$ as a bounded operator on $L^p(w\mu)$. Define 
\[
Rf(x)= \sum_{k=0}^\infty \frac{M^\mu_k f}{(2\|M^\mu\|_{L^p(w\mu)})^k}.
\]
Then $f(x)\le Rf(x)$, $\|Rf\|_{L^p(w\mu)}\le 2 \|f\|_{L^p(w\mu)}$, and $Rf$ is an $A_1(\mu)$ weight with constant $[Rf]_{A_1(\mu)}\le 2 \|M^\mu\|_{L^p(w\mu)}$. 
\end{lem}
\begin{proof}
The statements $f(x)\le Rf(x)$ and $\|Rf\|_{L^p(w\mu)}\le 2 \|f\|_{L^p(w\mu)}$ are obvious. Since 
\[
M^\mu(Rf)\le \sum_{k=0}^\infty \frac{M^\mu_{k+1} f}{(2\|M^\mu\|_{L^p(w\mu)})^k}\le 2\|M^\mu\|_{L^p(w\mu)} Rf,
\]
we have
\[
[Rf]_{A_1(\mu)}\le \sup_R \big(\inf_R M^\mu(Rf) \big)  \big(\operatornamewithlimits{ess\,inf}_R  Rf \big) ^{-1}\le 2\|M^\mu\|_{L^p(w\mu)}.
\]
We are done.
\end{proof}
With the above Rubio de Francia algorithm at hand, we are able to prove the bi-parameter version of \cite[Theorem 3.1]{LMO} and the corresponding endpoint cases similarly as in \cite[Theorem 2.3]{LMMOV}. On the other hand, the key technical lemma \cite[Lemma 2.14]{LMMOV} can be extended to the bi-parameter setting very easily. Using these as in \cite{LMMOV} we obtain Theorem \ref{thm:ext}.

The above Rubio de Francia algorithm, of course, also yields the following standard linear extrapolation.
Let $\mu\in A_\infty$ and assume that
\begin{equation}\label{eq:eq23}
\| g\|_{L^{p_0}(w\mu)}\lesssim \|f\|_{L^{p_0}(w\mu)}
\end{equation}
for all $w \in A_{p_0}(\mu)$. Then the same inequality holds for all $p \in (1,\infty)$ and $w \in A_p(\mu)$.
Using this and Lemma \ref{lem:lem5} we  obtain Proposition  \ref{prop:vecvalmax} via the following standard argument.

\begin{proof}[Proof of Proposition \ref{prop:vecvalmax}]
From Lemma 8.2 we directly have that
$$
\Big\| \Big( \sum_i (M^\mu f^i_j)^s \Big)^{\frac 1s} \Big\|_{L^s(w\mu)}
\lesssim \Big\| \Big( \sum_i |f^i_j|^s \Big)^{\frac 1s} \Big\|_{L^s(w\mu)}
, \quad w \in A_s(\mu).
$$
Then, the extrapolation described around \eqref{eq:eq23} gives that
$$
\Big\| \Big( \sum_i (M^\mu f^i_j)^s \Big)^{\frac 1s} \Big\|_{L^t(w\mu)}
\lesssim \Big\| \Big( \sum_i |f^i_j|^s \Big)^{\frac 1s} \Big\|_{L^t(w\mu)}
, \quad w \in A_t(\mu).
$$
This in turn gives
$$
\Big\| \Big( \sum_j \Big( \sum_i (M^\mu f^i_j)^s \Big)^{\frac ts}\Big)^{\frac 1t} \Big\|_{L^t(w\mu)}
\lesssim \Big\|\Big( \sum_j \Big( \sum_i |f^i_j|^s \Big)^{\frac ts}\Big)^{\frac 1t} \Big\|_{L^t(w\mu)}
, \quad w \in A_t(\mu).
$$
Extrapolating once more concludes the proof.
\end{proof}

\section{The multi-parameter case}\label{sec:multi}
One can approach the multi-parameter case as follows.
\begin{enumerate}
\item What is the definition of an SIO/CZO?
The important base case is the linear multi-parameter definition given in \cite{Ou}. That can
be straightforwardly extended to the multilinear situation as in Section \ref{sec:SIOs}.
The definition becomes extremely lengthy due to the large number of different partial kernel representations, and 
for this reason we do not write it down explicitly. However, there is no complication in combining the linear multi-parameter
definition \cite{Ou} and our multilinear bi-parameter definition.

We mention that another way to define the operators would be to adapt a Jour\'ne \cite{Jo} style definition 
-- this kind of vector-valued definition would be shorter to state.  In this paper we do not use the Journ\'e style formulation.
However, for the equivalence of the Journ\'e style definitions and the style we use here 
see \cite{Grau, LMV, Ou}.

\item Is there a representation theorem in this generality? Yes -- see Remark \ref{rem:mrem}.
The linear multi-parameter representation theorem
is proved in \cite{Ou}. The multilinear representation theorems \cite{AMV, LMV} are stated only in the bi-parameter setting for convenience.
However, using the multilinear methods from \cite{AMV, LMV} the multi-parameter theorem \cite{Ou} can easily be generalised to
the multilinear setting. 

\item How do the model operators look like? Studying the above presented bi-parameter model operators, one realises that the philosophies in each parameter
are independent of each other -- for example, if one has a shift type of philosophy in a given parameter, one needs at least two cancellative Haar functions in
that parameter. With this logic it is clear how to define the $m$-parameter analogues just by working parameter by parameter.
 Alternatively, one can take all possible $m$-fold tensor products
of one-parameter $n$-linear model operators, and then just replace the appearing product coefficients by general coefficients. This yields the form
of the model operators.

We demonstrate this with an example of a bilinear tri-parameter partial paraproduct. Tri-parameter
partial paraproducts have the shift structure in one or two of the parameters. In the remaining parameters
there is a paraproduct structure. The following is an example of a partial paraproduct with the shift structure in the first parameter and
the paraproduct structure in the second and third parameters:
\begin{equation*}
\begin{split}
&\sum_{K=K^1 \times K^2 \times K^3 \in \calD} 
\sum_{\substack{I_1^1, I^1_2, I_{3}^1 \in \calD^1 \\ (I^1_j)^{(k_j)}=K^1}}
\Big[a_{K,(I^1_j)} \\
&\hspace{1cm}\Big\langle f_1, h_{I_1^1} \otimes \frac{1_{K^2}}{|K^2|} \otimes \frac{1_{K^3}}{|K^3|}\Big\rangle 
\Big\langle f_2, h_{I^1_2}^0 \otimes \frac{1_{K^2}}{|K^2|} \otimes h_{K^3} \Big\rangle 
\Big\langle f_3, h_{I^1_3} \otimes  h_{K^2} \otimes \frac{1_{K^3}}{|K^3|} \Big\rangle\Big].
\end{split}
\end{equation*}
Here $\calD= \calD^1 \times \calD^2 \times \calD^3$. The assumption on the
coefficients is that when $K^1$, $I^1_1$, $I^1_2$ and $I^1_3$ are fixed, then
$$
\| (a_{K,(I^1_j)})_{K^2 \times K^3 \in \calD^2 \times \calD^3} \|_{\BMO_{{\rm{prod}}}} 
\le \frac{|I^1_1|^{ \frac 12}|I^1_2|^{\frac12} |I^1_3|^{\frac12}}{|K^1|^2}.
$$
In the shift parameter there is at least two cancellative Haar functions and in the paraproduct parameters 
there is exactly one cancellative Haar and the remaining functions are normalised  indicators. Thus, this is a generalization
of $S \otimes \pi \otimes \pi$, where $S$ is a one-parameter shift and $\pi$ is a one-parameter paraproduct -- and all model
operators arise like this.

\item Finally, is it more difficult to show the genuinely multilinear weighted estimates for the $m$-parameter, $m \ge 3$, model operators compared to the bi-parameter model operators? 
When it comes to shifts and full paraproducts, there is no essential difference -- their boundedness always reduces to Theorem \ref{thm:thm3}, which
has an obvious $m$-parameter version. With out current proof, the answer for the partial paraproducts is more complicated. Thus, we will
elaborate on how to prove the weighted estimates for $m$-parameter partial paraproducts. Notice that previously e.g. in \cite{LMV} we
could only handle bi-parameter partial paraproducts as our proof exploited the one-parameter nature of the paraproducts by sparse domination.
Here we have already disposed of sparse domination, but the proof is still complicated and leads to some new philosophies in higher parameters.
\end{enumerate}

We now discuss how to prove a tri-parameter analogue of Theorem \ref{thm:thm4}. We can have a partial paraproduct with a bi-parameter paraproduct
component and a one-parameter shift component, or the other way around. Regardless of the form, the initial stages of the proof of Theorem \ref{thm:thm4}
can be used to reduce to estimating the weighted $L^2$ norms of  certain functions which are analogous to \eqref{eq:eq24} and \eqref{eq:eq25}.
Most of these norms can be estimated with similar steps as in the bi-parameter case. However, also a new type of variant appears.
An example of such a variant is given by
\begin{equation}\label{eq:eq27}
F_{j,K^1}=1_{K^1} \sum_{(L^1_j)^{(l_j)}=K^1} \frac{|L^1_j|^{\frac 12 }}{|K^1|}\Big( \sum_{K^2} \frac{1_{K^2}}{|K^2|}
M_{\calD^3}^{\langle \sigma_j \rangle_{K^{1,2}}}\big(\langle f_j, h_{L^1_j} \otimes h_{K^2} \rangle \langle \sigma_j \rangle_{K^{1,2}}^{-1}\big)^2\Big)^{\frac 12},
\end{equation}
where $f_j \colon \R^{d_1} \times \R^{d_2} \times \R^{d_3} \to \C$.
The goal is to estimate $\sum_{K^1} \| F_{j,K^1} \|_{L^2(\sigma_j)}^2$. Here we are denoting $K^{1,2} = K^1 \times K^2$ the original tri-parameter
rectangle being $K = K^1 \times K^2 \times K^3$, and for brevity we only write $\langle \sigma_j \rangle_{K^{1,2}}$ instead of
$ \langle \sigma_j \rangle_{K^{1,2}, 1, 2}$.

Comparing with \eqref{eq:eq25}, the key difference is that
in \eqref{eq:eq25} the measure of the maximal function depended only on $K^1$. Here, it depends also on $K^2$, and therefore we have maximal functions
with respect to different measures inside the norms. We will use the following lemma and the appearing new type of extrapolation trick to overcome this.
\begin{lem}\label{lem:lem8}
Let $\mu \in A_\infty(\R^{d_1} \times \R^{d_2})$ be a bi-parameter weight. 
Let $\calD=\calD^1\times \calD^2$ be a grid of bi-parameter dyadic rectangles in $\R^{d_1} \times \R^{d_2}$.
Suppose that for each $m \in \Z$ and $K^1 \in \calD^1$ we have a function $f_{m,K^1} \colon \R^{d_2} \to \C$. Then,
for all $p,s,t \in (1, \infty)$, the estimate
$$
\Big\| \Big( \sum_{m } \Big( \sum_{K^1} 1_{K^1} 
M_{\calD^2}^{\langle \mu \rangle_{K^1}} (f_{m,K^1})^t \Big)^{\frac st} \Big)^{\frac 1s} \Big\|_{L^p(w\mu)}
\lesssim \Big\| \Big( \sum_{m } \Big( \sum_{K^1} 1_{K^1}
|f_{m,K^1}|^t \Big)^{\frac st} \Big)^{\frac 1s} \Big\|_{L^p(w\mu)}
$$
holds for all $w \in A_p(\mu)$.
\end{lem}

\begin{proof}
By extrapolation, see the discussion around \eqref{eq:eq23}, it suffices to take a function $f \colon \R^{d_2} \to \C$ and show that
$$
\big\| 1_{K^1} M_{\calD^2}^{\langle \mu \rangle_{K^1}} f \big\|_{L^q(w\mu)}
\lesssim \| 1_{K^1} f \|_{L^q(w\mu)}
$$
for some $q \in (1, \infty)$ and for all $w \in A_q(\mu)$. We fix some $w \in A_{q}(\mu)$. 
The above estimate can be rewritten as
$$
\big\| M_{\calD^2}^{\langle \mu \rangle_{K^1}} f \big\|_{L^q(\langle w\mu\rangle_{K^1})}|K^1|^{\frac 1q}
\lesssim \| f \|_{L^q(\langle w\mu\rangle_{K^1})}|K^1|^{\frac 1q}.
$$

We have the identity
\begin{equation}\label{eq:eq22}
\langle w\mu\rangle_{K^1}(x_2)
= \frac{\langle w\mu\rangle_{K^1}(x_2)}{\langle \mu \rangle_{K^1}(x_2)}\langle \mu \rangle_{K^1}(x_2)
= \langle w(\cdot,x_2) \rangle_{K^1}^{\mu(\cdot, x_2)}\langle \mu \rangle_{K^1}(x_2).
\end{equation}
Define $v(x_2) =\langle w(\cdot,x_2) \rangle_{K^1}^{\mu(\cdot, x_2)}$. We show that
$v \in A_q(\langle \mu \rangle_{K^1})$.
Let $I^2$ be a cube in $\R^{d_2}$. First, we have that
$$
\int_{I^2} v \langle \mu \rangle_{K^1}
= \int_{I^2} \langle w(\cdot,x_2) \rangle_{K^1}^{\mu(\cdot, x_2)}\langle \mu(\cdot,x_2) \rangle_{K^1} \ud x_2
=\int_{K^1 \times I^2}w \mu |K^1|^{-1}.
$$
Therefore,
$
\langle v \rangle_{I^2}^{\langle \mu \rangle_{K^1}}
=\langle w \rangle_{K^1 \times I^2 }^\mu.
$
H\"older's inequality gives that
$$
\big(\langle w(\cdot,x_2) \rangle_{K^1}^{\mu(\cdot, x_2)}\big)^{-\frac{1}{q-1}}
\le \big\langle w(\cdot,x_2)^{-\frac{1}{q-1}} \big \rangle_{K^1}^{\mu(\cdot, x_2)},
$$
which shows that
$$
\int_{I^2} v^{-\frac{1}{q-1}} \langle \mu \rangle_{K^1}
\le \int_{I^2}\big\langle w(\cdot,x_2)^{-\frac{1}{q-1}} \big \rangle_{K^1}^{\mu(\cdot, x_2)} \langle \mu(\cdot, x_2) \rangle_{K^1} \ud x_2
=\int_{K^1\times I^2} w^{-\frac{1}{q-1}} \mu |K^1|^{-1}.
$$
Thus, we have that
$$
\big(\langle v^{-\frac{1}{q-1}} \rangle_{I^2}^{\langle \mu \rangle_{K^1}}\big)^{q-1}
\le \big(\langle w^{-\frac{1}{q-1}} \rangle_{K^1 \times I^2}^{\mu}\big)^{q-1}.
$$
These estimates yield that $[v]_{A_q(\langle \mu \rangle_{K^1})} \le [w]_{A_q(\mu)}$.

Recall the identity \eqref{eq:eq22}. Since $\langle \mu \rangle_{K^1} \in A_\infty$, we have that
\begin{equation*}
\big\| M_{\calD^2}^{\langle \mu \rangle_{K^1}} f \big\|_{L^q(\langle w\mu\rangle_{K^1})}
=\big\| M_{\calD^2}^{\langle \mu \rangle_{K^1}} f \big\|_{L^q(v\langle \mu\rangle_{K^1})}
\lesssim \|  f \|_{L^q(v\langle \mu\rangle_{K^1}))}
=\|  f \|_{L^q(\langle w\mu\rangle_{K^1}))},
\end{equation*}
where we used Lemma \ref{lem:lem5}. This concludes the proof.
\end{proof}
We now show how to estimate \eqref{eq:eq27}. First, we have that $\|F_{j,K^1}\|_{L^2(\sigma_j)}^2$ is less than $ 2^{\frac{2l^1_jd_1}{s'}}$ multiplied by
\begin{equation*}
\bigg\| \bigg[\sum_{(L^1_j)^{(l_j)}=K^1} \bigg[\frac{|L^1_j|^{\frac 12 }}{|K^1|}\Big( \sum_{K^2} \frac{1_{K^2}}{|K^2|}
M_{\calD^3}^{\langle \sigma_j \rangle_{K^{1,2}}}
\big(\langle f_j, h_{L^1_j} \otimes h_{K^2} \rangle \langle \sigma_j \rangle_{K^{1,2}}^{-1}\big)^2\Big)^{\frac 12}\bigg]^{s} \bigg]^{\frac 1s} 
\bigg \|_{L^2(\langle\sigma_j\rangle_{K^1})}^2|K^1|.
\end{equation*}
The exponent $s \in (1, \infty)$ is chosen small enough so that we get a suitable dependence on the complexity
through $2^{\frac{2l^1_jd_1}{s'}}$, see the corresponding step in the bi-parameter case.
Since $\langle \sigma_j \rangle_{K^1} \in A_\infty(\R^{d_2} \times \R^{d_3})$, we can use Lemma \ref{lem:lem8} to have that the last term 
is dominated by
$$
\bigg\| \bigg[\sum_{(L^1_j)^{(l_j)}=K^1} \bigg[\frac{|L^1_j|^{\frac 12 }}{|K^1|}\Big( \sum_{K^2} \frac{1_{K^2}}{|K^2|}
\big|\langle f_j, h_{L^1_j} \otimes h_{K^2} \rangle \langle \sigma_j \rangle_{K^{1,2}}^{-1}\big|^2\Big)^{\frac 12}\bigg]^{s} \bigg]^{\frac 1s} 
\bigg \|_{L^2(\langle\sigma_j\rangle_{K^1})}^2|K^1|.
$$
After these key steps it only remains to use Proposition \ref{prop:prop3} twice in a very similar way as in the bi-parameter proof.
We are done.

\section{Applications}

\subsection{Mixed-norm estimates}
With our main result, Theorem \ref{thm:intro1}, and extrapolation, Theorem \ref{thm:ext}, the following result becomes immediate. 
\begin{thm}
Let $T$ be an $n$-linear $m$-parameter Calder\'on-Zygmund operator.
Let $1<p_i^j \le \infty$, $i = 1, \ldots, n$, with $\frac {1}{p^j}= \sum_i \frac{1}{p_i^j} >0$, $j = 1, \ldots, m$.
Then we have that
\[
\| T(f_1,\ldots, f_n)\|_{L^{p^1} \cdots L^{p^m}}\lesssim \prod_{i=1}^n \| f_i\|_{L^{p^1_i} \cdots L^{p^m_i}}.
\]
\end{thm}
\begin{rem}
We understand this as an a priori estimate with $f_i\in L_c^\infty$ -- this is only a concern when some $p_i^j$ is $\infty$. 
In \cite{LMMOV}, which concerned the bilinear bi-parameter case with \emph{tensor} form CZOs, we went to great lengths to check that
this restriction can always be removed. We do not want to get into such considerations here, and prefer this a priori
interpretation at least when $n \ge 3$. See also \cite{LMV} for some previous results for bilinear bi-parameter CZOs that are not of tensor form, but where, compared to \cite{LMMOV}, the range of exponents had some limitations in the $\infty$ cases. See also \cite{DO}.

We also mention that mixed-norm estimates for multilinear bi-parameter Coifman-Meyer operators have been previously
obtained in \cite{BM1} and \cite{BM2}.
Related to this, bi-parameter mixed norm Leibniz rules were proved in \cite{OW}. 

\end{rem}

The proof is immediate by extrapolating with tensor form weights. For the general idea see \cite[Theorem 4.5]{LMMOV} -- here
the major simplification is that everything can be done with extrapolation and the operator-valued analysis is not needed. This is because the weighted
estimate, Theorem \ref{thm:intro1}, is now with the genuinely multilinear weights unlike in \cite{LMMOV, LMV}.

\subsection{Commutators}
We will state these applications in the bi-parameter case $\R^d = \R^{d_1} \times \R^{d_2}$. The $m$-parameter versions are obvious.
We define 
$$
[b, T]_k(f_1,\ldots, f_n) := bT(f_1, \ldots, f_n) - T(f_1, \ldots, f_{k-1}, bf_k, f_{k+1}, \ldots, f_n).
$$
One can also define the iterated commutators as usual. We say that $b \in \operatorname{bmo}$ if
$$
\|b\|_{\operatorname{bmo}} = \sup_R \frac{1}{|R|}
\int_R |b-\ave{b}_R| < \infty,
$$
where the supremum is over rectangles.
Recall that given $b \in \operatorname{bmo}$, we have 
\begin{equation}\label{e:litbmo}
\|b\|_{\rm{bmo}}\sim \max\big(\esssup_{x_1\in \R^{d_1}} \|b(x_1,\cdot)\|_{\BMO(\R^{d_2})}, \esssup_{x_2\in \R^{d_2}} \|b(\cdot,x_2)\|_{\BMO(\R^{d_1})}\big).
\end{equation}
See e.g. \cite{HPW}.
In the one-parameter case the following was proved in \cite[Lemma 5.6]{DHL}.
\begin{prop}\label{prop:dhl}
Let $\vec p=(p_1,\dots, p_n)$ with $1<p_1,\ldots, p_n<\infty$ and $\frac 1p=\sum_i \frac 1{p_i} <1$. Let $\vec w = (w_1,\dots, w_n)\in A_{\vec p}$. Then for any $1\le j\le n$ we have
\[
\vec w_{b,z}:=(w_1,\dots, w_je^{\Re(bz)},\dots, w_n)\in A_{\vec p}
\]with $[\vec w_{b,z}]_{A_{\vec p}}\lesssim [\vec w]_{A_{\vec p}}$ provided that 
$$
|z|\le \frac{\epsilon}{\max([w^p]_{A_\infty}, \max_i [w_i^{-p_i'}]_{A_\infty}) \|b\|_{\BMO}},
$$
where $\epsilon$ depends on $\vec p$ and the dimension of the underlying space.
\end{prop}

If $1< p_1, \dots, p_n < \infty$, then there holds that
$[\vec w]_{A_{\vec p}(\R^{d})} < \infty$ if and only if
\[
\max \big(\esssup_{x_1\in \R^{d_1}}[\vec w(x_1,\cdot)]_{A_{\vec p}(\R^{d_2})},  \esssup_{x_2\in \R^{d_2}}[\vec w(\cdot,x_2)]_{A_{\vec p}(\R^{d_1})}\big) < \infty.
\]
Moreover, we have that the above maximum satisfies $\max (\cdot, \cdot) \le [\vec w]_{A_{\vec p}(\R^{d})}\lesssim \max (\cdot, \cdot)^{\gamma}$
where $\gamma$ is allowed to depend on $\vec p$ and $d$. The first estimate follows from the Lebesgue differentiation theorem. The second estimate can be 
proved by using Lemma \ref{lem:lem1} and the corresponding linear statement, see \eqref{eq:eq28}.
Using this, \eqref{e:litbmo} and Proposition \ref{prop:dhl} gives a bi-parameter version of Proposition \ref{prop:dhl} --  
the statement is obtained by replacing $\BMO$ with $\operatorname{bmo}$, and the quantitative estimate is of the form 
$[w_{b,z}]_{A_{\vec p}} \lesssim  [ \vec w ]_{A_{\vec p}}^\gamma$.
Now, we have everything ready to prove the following commutator estimate.
\begin{thm}
Suppose $T$ is an $n$-linear bi-parameter CZO in $\R^d = \R^{d_1} \times \R^{d_2}$,
$1 < p_1, \ldots, p_n \le \infty$ and $1/p = \sum_i 1/p_i> 0$.
Suppose also that $b \in \operatorname{bmo}$.
Then for all $1\le k\le n$ we have the commutator estimate
$$
\| [b, T]_k(f_1,\dots, f_n) w\|_{L^p} \lesssim \|b\|_{\operatorname{bmo}} \prod_{i=1}^n \|f_iw_i\|_{L^{p_i}}
$$
for all $n$-linear bi-parameter weights $\vec w = (w_1, \ldots, w_n) \in A_{\vec p}$.
Analogous results hold for iterated commutators.
\end{thm}
\begin{proof}
We assume $\|b\|_{\operatorname{bmo}} = 1$.
It suffices to study $[b, T]_1$, and in fact we shall prove the following principle. Once we have 
\[
\| T(f_1,\dots, f_n) w\|_{L^p}\lesssim \prod_{i=1}^n \|f_iw_i\|_{L^{p_i}}
\]
for some $\vec p$ in the Banach range, then 
\[
\| [b,T]_1(f_1,\dots, f_n) w\|_{L^p}\lesssim \prod_{i=1}^n \|f_iw_i\|_{L^{p_i}}.
\]
In this principle the form of the $n$-linear operator plays no role ($T$ does not need to be a CZO).
The iterated cases follow immediately from this principle and the full range then follows from extrapolation. 

Define 
\[
T_z^1(f_1, \dots, f_n)=e^{zb} T(e^{-zb}f_1, f_2,\dots, f_n).
\]
Then, by the Cauchy integral theorem, we get for nice functions $f_1,\dots, f_n$, that
\[
[b, T]_1(f_1,\dots, f_n)=\frac{\ud}{\ud z}T_z^1(f_1,\dots, f_n)\Big|_{z=0}
=\frac {-1}{2\pi i} \int_{|z|=\delta} \frac{T_z^1(f_1,\dots, f_n)}{z^2}\ud z,\quad \delta>0. 
\]
Since $p\ge 1$, by Minkowski's inequality
\[
\| [b, T]_1(f_1,\dots, f_n)w\|_{L^p}\le \frac 1{2\pi \delta^2}\int_{|z|=\delta} \|T_z^1(f_1,\dots, f_n) w\|_{L^p}|\ud z|.
\]
We choose 
$$
\delta
\sim \frac{1}{\max([w^p]_{A_\infty}, \max_i [w_i^{-p_i'}]_{A_\infty}) }.
$$
This allows to use the bi-parameter version of Proposition \ref{prop:dhl} to have that
\begin{align*}
\|T_z^1(f_1,\dots, f_n) w\|_{L^p} &=\|T(e^{-zb}f_1,f_2,\dots, f_n) we^{\Re(bz)}\|_{L^p}\\
&\lesssim \| e^{-zb} f_1w_1e^{\Re(bz)}\|_{L^{p_1}}\prod_{i=2}^n \|f_iw_i\|_{L^{p_i}}
=  \prod_{i=1}^n \|f_iw_i\|_{L^{p_i}}.
\end{align*}
The claim follows.
\end{proof}


\end{document}